\documentclass{amsart}

\usepackage{amsmath,amssymb,url}
\usepackage{enumitem} %%% Use the macros provided by this package when list labels stick out into the left margin. See the list in Section 7.

\usepackage[foot]{amsaddr}

\usepackage{amsthm}
\usepackage{mathrsfs}
\usepackage{tabu}
\usepackage{tikz}
\usepackage{caption}
\usepackage{indentfirst}
\usepackage{soul}
\usepackage{bussproofs}
\usepackage{proof} 
\usepackage{stmaryrd}
\usepackage{comment}
\usepackage{subfiles}
\usepackage{color}
\usepackage{hyperref}
\usepackage{mathtools}

\usetikzlibrary{arrows}
\usetikzlibrary{matrix}

\theoremstyle{plain}
\newtheorem{theorem}{Theorem}[section]
\newtheorem{prop}[theorem]{Proposition}
\newtheorem{corollary}[theorem]{Corollary}
\newtheorem{lemma}[theorem]{Lemma}

\theoremstyle{definition}
\newtheorem{remark}[theorem]{Remark}

\DeclareMathOperator*{\bigor}{OR}

\DeclareMathOperator*{\expo}{exp}
\DeclareMathOperator*{\primes}{primes}

\newcommand{\ldd}{{\bbslash}}
\newcommand{\rdd}{{\sslash}}
\newcommand{\m}[1]{{\mathbf {#1} }}
\newcommand{\bb}[1]{\mathbb {#1}}
\newcommand{\jn}{\vee}
\newcommand{\mt}{\wedge}
\newcommand{\ra}{\rightarrow}
\newcommand{\lra}{\leftrightarrow}

\newcommand{\rd}{{/}}
\newcommand{\ld}{{\setminus}}

\numberwithin{equation}{section}

\begin{document}

%%%%%%%%%%%%%%%%%%%%%%%%%%%%%%%%%%%%%%%%%%%%%%%%%%%%%%%%%%%%%%%%%%%%%%
%% FRONT MATTER
%%%%%%%%%%%%%%%%%%%%%%%%%%%%%%%%%%%%%%%%%%%%%%%%%%%%%%%%%%%%%%%%%%%%%%

\title[Unilinear residuated lattices]{Unilinear residuated lattices:\\
axiomatization, varieties and FEP}

%% For a single-authored paper, please use .
%% For multiple-authored papers, the line above MUST NOT be commented out
%% and $^*$ should be attached to the end of the name of the corresponding
%% author, as in the example below. 

%% The $^*$ on the name of the corresponding author on the submitted 
%% version of all multiple-authored papers is required by new EU data 
%% protection regulations. In line with the policy of Algebra Universalis 
%% since its inception, the $^*$ will not appear on the published version 
%% of the paper.

%% First author: in the order \author, \address, \urladdr, \email
%% For a sole-authored paper, use the \author[]{} command. 
\author{Nikolaos Galatos}
\address[N. Galatos]{Department of Mathematics\\
University of Denver\\Colorado\\USA}
%\urladdr{https://math.du.edu/~ngalatos/}
\email{ngalatos@du.edu}
%% For a multiple-authored paper, use the command \corrauthor[]{} for the corresponding author.
%% The corresponding author does not have to be the first-named author.
\author{Xiao Zhuang}
\address[Xiao Zhuang]{Department of Mathematics\\
University of Denver\\Colorado\\USA}
%\urladdr{}
\email[Corresponding author]{xiao.zhuang@du.edu}

%% Second author: in the order \author, \address, \urladdr, \email

%% Thanks (Optional)
\thanks{}

%% Dedication (Optional)
\dedicatory{}

%% AMS subject classification; see http://www.ams.org/msc
%% List classification codes in order of relevance

%% Key words and phrases

\begin{abstract}
We characterize all residuated lattices that have height equal to $3$ and show that the variety they generate has continuum-many subvarieties.
More generally, we study unilinear residuated lattices: their lattice is a union of disjoint incomparable chains, with bounds added.
We we give two general constructions of unilinear residuated lattices, provide an axiomatization and a proof-theoretic calculus for the variety they generate, and prove the finite model property for various subvarieties.
\end{abstract}

\maketitle

%%%%%%%%%%%%%%%%%%%%%%%%%%%%%%%%%%%%%%%%%%%%%%%%%%%%%%%%%%%%%%%%%%%%%%
%% MAIN MATTER
%%%%%%%%%%%%%%%%%%%%%%%%%%%%%%%%%%%%%%%%%%%%%%%%%%%%%%%%%%%%%%%%%%%%%%

\section{Introduction}

Residuated lattices generalize various well-known algebraic structures such as lattice-ordered groups, the ideals of a unital ring, and relation algebras, among others.
They also form algebraic semantics for various substructural logics, such as classical, intuitionistic, relevance, linear and many-valued logic; as a result further examples of residuated lattices include Boolean, Heyting, MV and BL-algebras.
We refer the reader to \cite{galatos2007residuated} for an introduction to residuated lattices and substructural logics.

A substantial amount of work has focused on the study of totally-ordered residuated lattices (residuated chains) and the variety they generate (semilinear residuated lattices).
Here, we start our study by exploring the other extreme: residuated lattices whose elements form an antichain, with two bounds added to obtain a lattice.
In Section~\ref{s: M}, we show that all residuated lattices of height 3 are precisely the ones consisting of two parts: a zero-cancellative monoid and a semigroup of at most three elements, and we specify the process for putting these two parts together.

In Section~\ref{s: axiom} we provide an axiomatization for the positive universal class of residuated lattices of height up to three and of the variety $\mathsf{M}$  it generates.
More generally, we consider the class $\mathsf{URL}$ of \emph{unilinear} residuated lattices: they are based on disjoint unions of incomparable chains with two additional bounds.
We axiomatize the positive universal class $\mathsf{URL}$ and the variety $\mathsf{SRL}$ of \emph{semiunilinear} residuated lattices it generates.
Moreover, we show that the finitely subdirectly irreducible members of $\mathsf{SRL}$ are precisely the unilinear ones.
In the particular case of $\mathsf{M}$, the simplicity of height-3 lattices directly gives the semisimplicity of $\mathsf{M}$, but we further show that the variety $\mathsf{bM}$, containing algebras on the expanded language that includes the bounds, is a discriminator variety.
We conclude the section with a discussion of the proof-theory of $\mathsf{SRL}$.
In particular we present a hypersequent calculus for $\mathsf{SRL}$ that enjoys the cut-elimination property, thus resulting in an analytic system for $\mathsf{SRL}$.

In Section~\ref{s: continuum} we show that there are continuum-many subvarieties of $\mathsf{M}$.
These are actually subvarieties of $\mathsf{CM_G}$, the variety generated by height-3 unilinear residuated lattices where the middle layer is an abelian group.
In fact we show that subvarieties of $\mathsf{CM_G}$ correspond to $\mathsf{ISP_U}$-classes of abelian groups and we further present a completely combinatorial characterization of the subvariety lattice of $\mathsf{CM_G}$ (without any reference to group theory).
We extend this characterization a little further, by allowing the middle layer of the residuated lattice to also include some semigroup elements, coming from the characterization in Section~\ref{s: M}.

Section~\ref{s: FEP} contains a proof of the finite embeddability propery (FEP) for the variety $\mathsf{CM_G}$, thus contrasting the complexity coming from the continuum-many subvarieties with the fact that the universal theory of $\mathsf{CM_G}$ is decidable.
We also establish the FEP for more subvarieties of $\mathsf{SRL}$, which do not have the height-3 restriction.
To be more precise, the FEP holds for every subvariety of $\mathsf{SRL}$ that is axiomatized by equations in the language of multiplication, join and $1$, and satisfies any weak commutativity axiom and any knotted rule; we establish this result by using the method of residuated frames.

Finally, in Section~\ref{s: compact}, we focus our attention on unilinear residuated lattices $\m R$ where $M: = R \setminus \{\bot, \top\}$ is a submonoid and the bounds are absorbing with respect to the elements of $M$; we call such unilinear residuated lattices \emph{compact}.
We provide two constructions of compact residuated lattices, with the first one coming from a finite cyclic monoid.
In the second one $M$ is the cartesian product of a residuated chain and a cancellative monoid, relative to a 2-cocycle; thus it is a generalization of the semidirect product of monoids. 

\medskip

We continue with some preliminaries on residuated lattices.
A \emph{residuated lattice} is an algebra $(R, \wedge, \vee, \cdot, \backslash, /, 1)$ where
\begin{itemize}
    \item $(R, \wedge, \vee)$ is a lattice,

    \item $(R, \cdot, 1)$ is a monoid, and

    \item $xy \leq z$ iff $y \leq x \backslash z$ iff $x \leq z/y$ for all $x, y, z \in R$.
\end{itemize}

The last condition above is called \emph{residuation}.
Given posets $\mathbf{P}$ and $\mathbf{Q}$, a map $f: \mathbf{P} \rightarrow \mathbf{Q}$ is said to be \emph{residuated} if there exists a map $f^*: \mathbf{Q} \rightarrow \mathbf{P}$ such that
\[
f(x) \leq y \text{ iff } x \leq f^*(y)
\]
for all $x \in P$, $y \in Q$.

The following result is folklore in the theory of residuated maps.

\begin{lemma} \label{res.equivalence}
    A function $g$ from a poset $\mathbf{P}$ to a poset $\mathbf{Q}$ is residuated if and only if the set $\{x \in P: g(x) \leq y\}$ has a maximum for all $y \in Q$ and $g$ is order-preserving.
\end{lemma}

\begin{proof}
    Let $S_y = \{x \in P: g(x) \leq y\}$ and we assume that $g$ is residuated with residual $g^*$. 
    Note that $g^*(y) \leq g^*(y)$ yields $g(g^*(y)) \leq y$ so $g^*(y) \in S_y$.
    Also, for all $x \in S_y$, $g(x) \leq y$ hence $x \leq g^*(y)$.
    Therefore, $g^*(y) = \max S_y$.

    If $x_1 \leq x_2$, then since $g(x_2) \leq g(x_2)$ yields $x_2 \leq g^*(g(x_2))$, we get $x_1 \leq g^*(g(x_2))$; hence $g(x_1) \leq g(x_2)$.
    Therefore, $g$ is order-preserving.

    Now suppose $S_y$ has a maximum for all $y \in Q$ and $g$ preserves the order.
    We define $g^*: Q \rightarrow P$ by $g^*(y) = \max S_y$; clearly $g^*$ is order-preserving.
    If $g(x) \leq y$ for some $x \in P$, $y \in Q$, then $x \in S_y$ and $x \leq g^*(y)$ by definition.
    Conversely, if $x \leq g^*(y)$, then $g(x) \leq g(g^*(y))$ since $g$ is order-preserving.
    Moreover, $g^*(y) \in S_y$ so $g(g^*(y)) \leq y$; thus $g(x) \leq y$.
\end{proof}

We mention that if the assumption that $\{x \in P: g(x) \leq y\}$ has a maximum is replaced by the demand that it has a join, then the order-preservation of $g$ is not enough to give residuation.

Note that a lattice-ordered monoid supports a residuated lattice iff left and right multiplication are residuated.
So Lemma~\ref{res.equivalence} yields the following fact.

\begin{corollary}\label{c: RLmax}
    A lattice-ordered monoid $\mathbf{R}$ is a reduct of a residuated lattice iff multiplication is order-preserving and for all $x, z \in R$, the sets $\{y \in R: \, xy \leq z\}$ and $\{y \in R: \, yx \leq z\}$ have maximum elements.
    In such a case the expansion to a residuated lattice is unique by $x \ld z = \max \{y \in R: xy \leq z\}$ and $z \rd x = \max \{y \in R: yx \leq z\}$.
\end{corollary}

\begin{corollary}[Cor~3.12 of \cite{galatos2007residuated}]\label{c: RLjoin}
    A complete lattice-ordered monoid $\mathbf{R}$ is a reduct of a residuated lattice iff multiplication distributes over arbitrary joins.
\end{corollary}

In particular, multiplication distributes over the empty join, if it exists; so if there is a bottom element $\bot$, then $x \cdot  \bot=\bot = \bot \cdot x$, for all $x$.
For convenience, we set $x \ldd z := \{y \in R: xy \leq z\}$ and $z \rdd x := \{y \in R: yx \leq z\}$ for $x, z \in R$.

\begin{remark}\label{bottopdivison}
    Let $\mathbf{P} = (P, \wedge, \vee, \cdot, \bot, \top)$ be a bounded lattice-ordered semigroup.
    Since $\bot x = \bot$ for all $x \in P$, we have $\bot \ldd x = P$, so $\bot \backslash x = \max \bot \ldd x = \top$ for all $x \in P$.
    Also, since $x \ldd \top = P$, $x \ld \top = \top$ for all $x \in P$.
    Similarly, $x \rd \bot = \top$ and $\top \rd x = \top$ for all $x \in P$.
\end{remark}

A residuated lattice with bounds $\bot$ and $\top$ is called \emph{rigorously compact} if $\top x = x \top = \top$ for all $x \not = \bot$.
In this case we also have that $xy = \bot \Rightarrow x = \bot \text{ or } y = \bot$, since otherwise we get $x \not = \bot \not = y$, so $\bot = \bot \top = xy \top = x \top = \top$, a conradiction. 
Note that in rigorously compact residuated lattices we have $\bot \ld x = x \rd \bot = \top = x \ld \top = \top \rd x$, $\top \ld y = y \rd \top = \bot$, $z \ld \bot = \bot = \bot \rd z$ for all $x \in R$, $y \not = \top$, $z \not = \bot$.

%%%%%%%%%%%%%%%%%%%%%%%%%%%%%%%%%%%%%%%%%%%%%%%%%%%%%%%%%%%%%%%%%%%%%%
\section{Residuated Lattices on \texorpdfstring{$\mathbf{M}_X$}{Mx}}\label{s: M}

Residuated lattices based on chains have been studied extensively.
We start by looking into residuated lattices based on an antichain, with extra top and bottom elements.

\subsection{Properties}

Given a set $X$, we denote by $\mathbf{M}_X$ the lattice over the set $X \cup \{\bot, \top\}$, where $\top$ is the top element, $\bot$ is the bottom element, and $x \vee y = \top$ and $x \wedge y=\bot$, for distinct $x,y \in X$. 

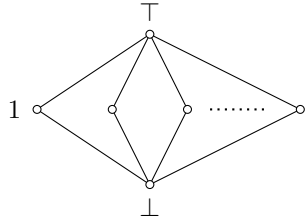
\begin{figure}[ht]
\centering
\begin{tikzpicture}
\draw (1.5, 1) -- (0, 0) -- (1.5, -1);
\draw (1.5, 1) -- (1, 0) -- (1.5, -1);
\draw (1.5, 1) -- (2, 0) -- (1.5, -1);
\draw (1.5, 1) -- (3.5, 0) -- (1.5, -1);
\draw[thick, dotted] (2.3, 0) -- (3.05, 0);

\filldraw [color = black, fill = white] (1.5, 1) circle (1.5pt)
(1.5, 1.3) node {$\top$};
\filldraw [color = black, fill = white] (1.5, -1) circle (1.5pt)
(1.5, -1.3) node {$\bot$};
\filldraw [color = black, fill = white] (0, 0) circle (1.5pt)
(-0.3, 0) node {$1$};
\filldraw [color = black, fill = white] (1, 0) circle (1.5pt);
\filldraw [color = black, fill = white] (2, 0) circle (1.5pt);
\filldraw [color = black, fill = white] (3.5, 0) circle (1.5pt);
\end{tikzpicture}
\caption{A residuated lattice over $M_X$}
\end{figure}

The characterization of all residuated lattices based on $\mathbf{M}_X$ where $X$ is non-empty and closed under multiplication is known (\cite{galatos2007residuated} p. 205): $X$ is a cancellative monoid, $\bot$ is absorbing in $M_X$ and $\top$ is absorbing in $X \cup \{\top\}$.
We will characterize all residuated lattices based on $\mathbf{M}_X$, even when $X$ is not closed under multiplication.

Recall that in every bounded residuated lattice the bottom element is absorbing.
Also, in a residuated lattice based on $\mathbf{M}_X$ we have $\top x, x \top \in \{x, \top\}$ for all $x$, since $1 \leq \top$ implies $x \leq \top x$ and $x \leq x \top$. 

In a residuated lattice $\bf{R}$ on $\mathbf{M}_X$, we define
\begin{align*}
    U_R = \{x \in R \setminus \{\bot, \top\}: x \top = \top\} \, \text{ and } \, Z_R = \{x \in R \setminus \{\bot, \top\}: x \top = x\},
\end{align*}
the set of elements that behave as units for $\top$ and the set of elements that behave as zeros for $\top$; when the residuated lattice is clear from the context we drop the subscript in $U_R$ and $Z_R$.
Note that $1 \in U$ and $U \cap Z = \emptyset$.

A monoid $\mathbf{S}$ with a zero (absorbing element) $0$ is called \emph{$0$-cancellative} if for all $x, y, z \in S$,
\begin{align*}
xy = xz \neq 0 & \; \Rightarrow \; y = z\\
yx = zx \neq 0 & \; \Rightarrow \; y = z.
\end{align*}

An element $c$ in a residuated residuated $\mathbf{R}$ lattice is called \emph{central} if $xc = cx$, for all $x \in R$.
Also, we denote by $\sqcup$ the disjoint union operation.

\begin{theorem}\label{t:deconstruction}
If $\mathbf{R}$ is a residuated lattice based on $\mathbf{M}_X$, then
\begin{enumerate}
    \item $\top$ is central in $\mathbf{R}$ and $R = U \sqcup Z \sqcup \{\bot, \top\}$.
    
    \item $U_{\top} = U \cup \{\top\}$ is a $\top$-cancellative submonoid of $\mathbf{R}$.

    \item $Z_{\bot} = Z \cup \{\bot\}$ is a subsemigroup of $\mathbf{R}$ with zero $\bot$, $|Z_{\bot}| \leq 3$ and $xy = \bot$ for all distinct $x, y \in Z_{\bot}$.\\    
    Also, either $Z_{\bot}$ is idempotent, or $Z_{\bot} = \{b, \bot\}$ with $b^2 = \bot$.
    
    \item $ab = ba = b$ for all $a \in U$ and $b \in Z$.	
\end{enumerate}
\end{theorem}

\begin{proof}

(1) We will show that $\top x = x \top$, for all  $x \in R$.
If $x$ is $\top, \bot$ or $\top$, then $\top x$ and $x \top$ both are equal to $\top, \bot, x$, respectively.
Also, if $x$ is incomparable to $1$, then $x \vee 1 = \top$, so $\top x = (1 \vee x) x = x \vee x^2 = x (1 \vee x) = x \top$.

Since $\top$ is central and $x \top \in \{x, \top\}$ for all $x$, we have that for every $x \in R \setminus \{\bot, \top\}$ either $x\in U$ or $x\in Z$, but not both.

(2) If $a, b \in U_{\top}$, then $ab \cdot \top = a \cdot b \top = a \top = \top$, so $ab \in U_{\top}$.
Similarly, $ba \in U_{\top}$ and $\top$ is a zero for $U_{\top}$.

If $x, y, z \in U_{\top}$ and $xy = xz \neq \top$, then $x (y \vee z) = xy \vee xz = xy \neq \top$.
So $y \vee z \not = \top$, because $x \top = \top$; in particular, $y \neq \top \neq z$.
Also, since $y, z \in U_{\top}$ and $\bot \not \in U_{\top}$, we get $y \neq \bot \neq z$; hence $y,z \in X$ and $y \vee z \not = \top$.
Since, $\mathbf{R}$ is based on $\mathbf{M}_X$, we get that $y = z$.
Similarly, we obtain the other implication of $\top$-cancellativity. 

(3) If $c, d \in Z_{\bot}$, then $cd \cdot \top = c \cdot d \top = cd$.
Also, $cd \leq c \top = c < \top$; hence $cd \in Z_{\bot}$.
Clearly, $\bot$ is a zero for $Z_{\bot}$.

Since $Z_{\bot} \subseteq X \cup \{\bot\}$, for distinct $x, y \in Z_{\bot}$, we have $xy = xy \wedge xy \leq x \top \wedge \top y = x \wedge y = \bot$.
So, if there were distinct $x, y, z \in Z$, then $y \vee z = \top$ and $x = x \top = x (y \vee z) = xy \vee xz = \bot \vee \bot = \bot$, a contradiction.
Therefore $|Z_{\bot}| \leq 3$.

If $b$ is a non-idempotent element of $Z_{\bot} \subseteq X \cup \{\bot\}$, then $b \not = \bot$ and $b^2 \leq b \top = b$, so $b^2=\bot$.
If $c$ is an element of $Z_{\bot}$ distinct from $b$ and $\bot$, then $b^2 = b^2 \vee \bot = b^2 \vee bc = b (b \vee c) = b \top = b$, a contradiction.
So, if $Z_{\bot}$ is not idempotent, then $Z_{\bot} = \{b, \bot\}$ and $b^2=\bot$.

(4) For $a \in U$ and $b \in Z$, using the centrality of $\top$, we get
$$b = \top b = \top a \cdot b = \top \cdot ab = ab \cdot \top = a \cdot b \top = ab.$$
Similarly, we get $ba = b$.
\end{proof}

It is straight-forward to see that that the possible options for the subsemigroup $Z_{\bot}$, mentioned in Theorem~\ref{t:deconstruction}(3) are precisely the ones in Figure~\ref{f:4tables}.
\begin{figure}[ht]
\begin{center}
\begin{tabular}{c | c}
 & $\bot$\\
\hline
$\bot$ & $\bot$ 
\end{tabular}
,
\begin{tabular}{c | c c}
 & $\bot$ & $b$\\
\hline
$\bot$ & $\bot$ & $\bot$\\
$b$ & $\bot$ & $b$
\end{tabular}
,
\begin{tabular}{c | c c}
 & $\bot$ & $b$\\
\hline
$\bot$ & $\bot$ & $\bot$\\
$b$ & $\bot$ & $\bot$
\end{tabular}
,
\begin{tabular}{c | c c c}
 & $\bot$ & $b_1$ & $b_2$\\
\hline
$\bot$ & $\bot$ & $\bot$ & $\bot$\\
$b_1$ & $\bot$ & $b_1$ & $\bot$\\
$b_2$ & $\bot$ & $\bot$ & $b_2$ 	
\end{tabular}
.
\end{center}
\caption{Four multiplication tables}
\label{f:4tables}
\end{figure}

Note that if a residuated lattices based on $\mathbf{M}_X$ is integral (i.e., it satisfies $x \leq 1$), then $U = \emptyset$.
By taking into account all of the possibilities for $Z_{\bot}$, it follows that the only integral residuated lattices based on $\mathbf{M}_X$ are the 2-element and 4-element Boolean algebras, the 3-element Heyting algebra and the 3-element MV-algebra.
The latter two, together with the $3$-element Sugihara monoid, are the only $3$-element residuated chains.

\subsection{Construction and characterization}

We now prove the converse of Theorem~\ref{t:deconstruction}.
Let $\mathbf{A}$ be a $\top$-cancellative monoid with zero $\top$ and $\mathbf{B}$ a semigroup with zero $\bot$, whose multiplication table is one of those in Figure~\ref{f:4tables}.

We define the lattice structure $\mathbf{M}_X$ on the set $R = A \cup B$, where $X = R \setminus \{\bot, \top\}$, $\bot$ is the bottom and $\top$ is the top. 
Also, we define a multiplication on $R$ that extends the multiplications on $\mathbf{A}$ and $\mathbf{B}$ by:
$xy = yx = y$,
for all $x \in A$ and $y \in B$.
We denote by $\mathbf{R}_{\mathbf{A}, \mathbf{B}}$ the resulting algebra.

\begin{theorem} \label{thmM_n}
If $\mathbf{A}$ is a $\top$-cancellative monoid with zero $\top$ and $\mathbf{B}$ is a semigroup with zero $\bot$, whose multiplication table is one of those in Figure~\ref{f:4tables}, then $\mathbf{R}_{\mathbf{A}, \mathbf{B}}$ is the reduct of a residuated lattice based on $\mathbf{M}_X$, where $X = (A \cup B) \setminus \{\bot, \top\}$.
\end{theorem} 

\begin{proof}
Since associativity holds in $\mathbf{A}$ and $\mathbf{B}$ and every element of $B$ is an absorbing element for $A$, we get that multiplication on $\mathbf{R}$ is associative.

Corollary~\ref{c: RLjoin} ensures that an expansion of $\mathbf{M}_X$ by a monoid structure is a residuated lattice iff multiplication distributes over arbitrary joins.
Since $\bot x = x \bot = \bot$ for all $x \in R$, multiplication distributes over the empty join.
Also, we observe every infinite join is equivalent to a finite join, so it suffices to show $x(y \vee z) = xy \vee xz$ and $(y \vee z)x = yx \vee zx$ for all $x, y, z \in R$ and $y \neq z$.
Here we prove $x(y \vee z) = xy \vee xz$.

If $\bot \in \{x, y, z\}$, then it is easy to check that this equation always holds, so we will assume that $\bot \notin \{x, y, z\}$.
Since $y \neq z$, we get $y \vee z = \top$.
Now we will verify that $x \top = xy \vee xz$.

If $x \in B$, then the left-hand side is $x$.
If, further, $y \in A$ or $z \in A$, then the right-hand side is $x \vee xz = x$ or $xy \vee x = x$, since $xu \leq x$ for all $u \in R$.
If $y, z \in B$, then since $|B| \leq 3$ and $y, z, \bot$ are distinct, we get $B = \{y, z, \bot\}$ and $x = y$ or $x = z$.
In this case, $xy \vee xz = x \vee \bot = x$, so the equation holds.

If $x \in A$, then the left-hand side is equal to $\top$.
If $y \in B$ and $z \in B$, then the right-hand side is $y \vee z = \top$, since $y \neq z$.
If $y \in B$ and $z \in A$, then the right-hand side is $y \vee xz = \top$, since $y \in B$, $xz \in A$ and $\bot \notin \{x, y, z\}$.
Likewise, if $y \in A$ and $z \in B$, then the right-hand side is $\top$.
If $y \in A$ and $z \in A$, then the right-hand side is $xy \vee xz = \top$ since $\mathbf{A}$ is $\top$-cancellative.

Similarly, we can show $(y \vee z)x = yx \vee xz$ for all $x, y, z \in R$.
\end{proof}

By Corollary~\ref{c: RLmax} the divisions are uniquely determined by the equations $x \ld z = \max \{y \in R: xy \leq z\}$ and $z \rd x = \max \{y \in R: yx \leq z\}$, and we give the precise values below.

It turns out that $A \cup \{\bot\}$ and $B \cup \{\top\}$ are subalgebras of $\mathbf{R}_{\mathbf{A}, \mathbf{B}}$.
In particular, $B \cup \{\top\}$ is the $2$-element Boolean algebra, the $3$-element Heyting algebra, $3$-element MV-algebra, or the $4$-element Boolean algebra, corresponding to the tables in Figure~\ref{f:4tables}.
The divisions are given by Remark~\ref{bottopdivison} and
\begin{align*}
    a_1 \backslash a_2 =
    \begin{cases}
        a_3 & \text{ if } a_1 a_3 = a_2\\
        \bot & \text{ otherwise}
    \end{cases}
    \quad
    a_2 / a_1 =
    \begin{cases}
        a_3 & \text{ if } a_3 a_1 = a_2\\
        \bot & \text{ otherwise}
    \end{cases}
\end{align*}
for $a_1, a_2, a_3 \in A$, where the $a_3$ is guaranteed to be unique, when it exists.
Finally, for $a \in A \setminus \{\top\}$ and $b \in B$, any operation between $a$ and $b$ works the same as the operation between $1$ and $b$.
For example, $b \backslash a = b \backslash 1$, $a \wedge b = 1 \wedge b$, $ab = 1b$, etc. 

By combining Theorem~\ref{t:deconstruction} and Theorem~\ref{thmM_n}, we obtain the following characterization.

\begin{corollary}\label{Char_of_RL_on_M_X}
The residuated lattices based on $\mathbf{M}_X$ are precisely the ones of the form $\mathbf{R}_{\mathbf{A}, \mathbf{B}}$, where $\mathbf{A}$ is a $\top$-cancellative monoid with zero $\top$ and $\mathbf{B}$ is a semigroup with zero $\bot$, whose multiplication table is one of those in Figure~\ref{f:4tables}.
\end{corollary} 

%%%%%%%%%%%%%%%%%%%%%%%%%%%%%%%%%%%%%%%%%%%%%%%%%%%%%%%%%%%%%%%%%%%%%%
\section{Axiomatizations}\label{s: axiom}
In this section we will provide axiomatizations for the various classes we will be considering and also discuss their proof theory.

\subsection{Axiomatization of residuated lattices based on \texorpdfstring{$\mathbf{M}_X$'s}{Mx}}

We start by giving an axiomatization for the variety $\mathsf{M}$ generated by all residuated lattices based on $\mathbf{M}_X$, where $X$ is a set; see Corollary~\ref{c: variety M_X}.
Since the lattice $\mathbf{M}_X$ is simple, when $|X| \geq 3$, residuated lattices based on $\mathbf{M}_X$ are also simple; if $|X| \leq 3$ the residuated lattice is simple, as well.
It turns out (Corollary~\ref{c: FSI variety M_X}) that these are precisely the subdirectly irreducible algebras in $\mathsf{M}$ and we will provide an axiomatization for them.

Actually, we can also expand the language of residuated lattices to include constants which then evaluate as bounds.
A \emph{bounded residuated lattice} is an expansion of a residuated lattice that happens to be based on a bounded lattice, by the addition of constants $\bot$ and $\top$, evaluating at these bounds (so $\bot \leq x \leq \top$, for all $x$).
We will consider both cases where the language includes the bounds or not, but opt for the axioms to be expressible without the need for bounds.
We can arrange for the axioms we will be considering to be positive universal sentences, which is convenient for applying the correspondence provided in \cite{galatos2004equational}.

A (bounded) residuated lattice is called \emph{unilinear} if it satisfies: 
\begin{equation}\tag{URL}\label{axiom_URL}
    \forall u_1, u_2, z, w \, \, (u_1 \leq u_2 \text{ or } u_2 \leq u_1 \text{ or } (u_1 \wedge u_2 \leq w \text{ and } z \leq u_1 \vee u_2))
\end{equation}

Note that a residuated lattice is unilinear iff it is linear or else the lattice is actually bounded and every pair of incomparable elements join to the top of the lattice and meet to the bottom of the lattice.
In other words the non-linear residuated lattices consist of two bounds and the rest of the lattice is a disjoint union of totally incomparable chains; see Figure~\ref{unilinear}.
For these non-linear unilinear residuated lattices, we will be denoting these bounds by $\bot$ and $\top$, even when the language does not include constants for the bounds.
We denote by $\mathsf{URL}$ and $\mathsf{bURL}$ the (positive universal) classes of unilinear and bounded unilinear residuated lattices, respectively.
Clearly, (bounded) residuated lattices on an $\mathbf{M}_X$ are unilinear.

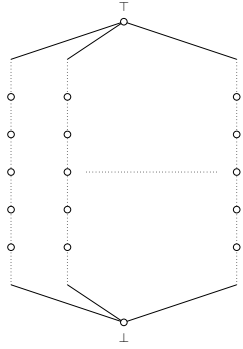
\begin{figure}[ht]
\scalebox{0.5}{
\centering
    \begin{tikzpicture}
        \draw (3, 4) -- (0, 3);
        \draw[dotted] (0, 3) -- (0, -3);
        \draw (0, -3) -- (3, -4);
        \draw (3, 4) -- (1.5, 3);
        \draw[dotted] (1.5, 3) -- (1.5, -3);
        \draw (1.5, -3) -- (3, -4);
        \draw[dotted] (2, 0) -- (5.5, 0);
        \draw (3, 4) -- (6, 3);
        \draw[dotted] (6, 3) -- (6, -3);
        \draw (6, -3) -- (3, -4);

        \filldraw [color = black, fill = white] (3, 4) circle (2.5pt)
            (3, 4.4) node {$\top$};
        \filldraw [color = black, fill = white] (3, -4) circle (2.5pt)
            (3, -4.4) node [black] {$\bot$};
        \filldraw [color = black, fill = white] (0, 2) circle (2.5pt)
            (-0.3, 2.3) node [black] {};
        \filldraw [color = black, fill = white] (0, 1) circle (2.5pt)
            (-0.3, 1.3) node [black] {};
        \filldraw [color = black, fill = white] (0, 0) circle (2.5pt)
            (-0.3, 0.3) node [black] {};
        \filldraw [color = black, fill = white] (0, -1) circle (2.5pt)
            (-0.4, -0.7) node [black] {};
        \filldraw [color = black, fill = white] (0, -2) circle (2.5pt)
            (-0.4, -1.7) node [black] {};

        \filldraw [color = black, fill = white] (1.5, 2) circle (2.5pt);
        \filldraw [color = black, fill = white] (1.5, 1) circle (2.5pt);
        \filldraw [color = black, fill = white] (1.5, 0) circle (2.5pt)
            (1.2, 0.3) node {};
        \filldraw [color = black, fill = white] (1.5, -1) circle (2.5pt);
        \filldraw [color = black, fill = white] (1.5, -2) circle (2.5pt);

        \filldraw [color = black, fill = white] (6, 2) circle (2.5pt);
        \filldraw [color = black, fill = white] (6, 1) circle (2.5pt);
        \filldraw [color = black, fill = white] (6, 0) circle (2.5pt)
            (6.3, 0.3) node [black] {};
        \filldraw [color = black, fill = white] (6, -1) circle (2.5pt);
        \filldraw [color = black, fill = white] (6, -2) circle (2.5pt);
    \end{tikzpicture}
}
\caption{A non-linear unilinear residuated lattice}
\label{unilinear}
\end{figure}

What distinguishes $\mathbf{M}_X$ from other lattices is its height, so we axiomatize unilinear residuated lattices whose height is no greater than a given number.
We are careful to formulate the first-order sentence so it has no implication in it and it remains a positive sentence.

\begin{prop}
    Given a natural number $n$, a (bounded) unilinear residuated lattice has height at most $n$ if and only if it satisfies
        \begin{equation}\tag{$h_n$}\label{axiom_of_finite_height}
            \forall x_1, \dots, x_{n+1} \, (\bigor \limits_{1 \leq m \leq n}
            x_1 \vee \dots \vee x_{m} = x_1 \vee \dots \vee x_{m+1}).
        \end{equation}
    Also, it has width at most $n$ if and only if it satisfies
        \begin{equation}\tag{$w_n$}\label{axiom_of_finite_width}
            \forall x_1, \dots, x_{n+1} \, (\bigor \limits_{1 \leq i \neq j \leq n+1} x_i \leq x_j).
        \end{equation}
\end{prop}

\begin{proof}
Having height at most $n$ is equivalent to saying that every subchain has at most $n$ elements.
Now, every subchain always has the form $a_1 \leq a_1 \vee a_2 \leq a_1 \vee a_2 \vee a_3 \leq \dots \leq a_1 \vee \dots \vee a_{k}$, where $a_1, \ldots, a_k$ are elements of the lattice and where the number of the inequalities that are equalities determines the number of elements in the chain.
So, having height at most $n$ is equivalent to stipulating that in every chain $a_1, a_1 \vee a_2, \ldots , a_1 \vee \dots \vee a_{n+1}$, at least two adjacent elements are equal.

Having width at most $n$ is equivalent to having at most $n$ pairwise incomparable elements.
\end{proof}

We denote by $\mathsf{URL}_n$ the subclass of $\mathsf{URL}$ axiomatized by (\ref{axiom_of_finite_height}).
In particular, $(h_3)$ is the universal closure (which we often suppress) of 
$$x_1 = x_1 \vee x_2 \text{ or } x_1 \vee x_2 = x_1 \vee x_2 \vee x_3 \text{ or } x_1 \vee x_2 \vee x_3 = x_1 \vee x_2 \vee x_3 \vee x_4.$$

\begin{corollary}
The (bounded) residuated lattices that are based on $\mathbf{M}_X$, for some $X$, together with the trivial algebra, are precisely the ones in the class $\mathsf{URL}_3$ ($\mathsf{bURL}_3$).
\end{corollary}

\subsection{Equational basis for \texorpdfstring{$\mathsf{M}$}{M}}

The class $\mathsf{URL}_3$ is axiomatized by positive universal sentences.
We note that \cite{galatos2004equational} provides a general method for axiomatizing the variety of residuated lattices generated by a positive universal class.
In detail, if $$1 \leq p_1 \text{ or } \cdots \text{ or } 1 \leq p_n$$ is a positive universal formula, then the variety generated by the residuated lattices satisfying the universal closure of the formula is axiomatized by the infinitely many equations 
$$1 = \gamma_1(p_1) \vee \cdots \vee \gamma_n (p_n)$$
where $\gamma_1, \ldots,\gamma_n \in \Gamma(Var)$, the set of all iterated conjugates.
The \emph{left conjugate} of $a$ by $x$ is the term $x \ld ax \mt 1$ and the \emph{right conjugate} is $xa \rd x \mt 1$; \emph{iterated conjugates} are obtained by repeated applications of left and right conjugates by various conjugating elements from the set $Var$ of variables.
If $\phi$ is a set of positive universal formulas, we denote by $\mathsf{V}_\phi$ the variety axiomatized by the set $\Gamma_\phi$ of all the equations corresponding to the positive universal formulas in $\phi$. 

We consider the variety $\mathsf{SRL}$ generated by the class $\mathsf{URL}$ and we call its elements \emph{semiunilinear}.
Since $\mathsf{URL}$ is axiomatized by
\begin{align*}
    x \leq y \text{ or } y \leq x \text{ or } (x \wedge y \leq z \text{ and } w \leq x \vee y ),
\end{align*}
which can be written as the conjunction of the two sentences
\begin{align*}
    x \leq y \text{ or } y \leq x \text{ or } x \wedge y \leq z, \qquad
        x \leq y \text{ or } y \leq x \text{ or } w \leq x \vee y,
\end{align*}
and, in turn, as
\begin{align*}
    1 \leq x \backslash y \text{ or } 1 \leq y \backslash x \text{ or } 1 \leq (x \wedge y) \backslash z, \;\;
    1 \leq x \backslash y \text{ or } 1 \leq y \backslash x \text{ or }  1 \leq w \backslash (x \vee y),
\end{align*}
we get the following result.

\begin{corollary}\label{c: SRLax}
The variety $\mathsf{SRL}$ of semiunilinear residuated lattices is axiomatized by the infinitely many equations 
$$1 = \gamma_1(x \backslash y) \vee \gamma_2(y \backslash x) \vee \gamma_3 ((x \wedge y) \backslash z) \qquad 1 = \gamma_4(x \backslash y) \vee \gamma_5(y \backslash x) \vee \gamma_6 (w \backslash (x \vee y)),$$
where $ \gamma_1, \gamma_2, \gamma_3, \gamma_4, \gamma_5, \gamma_6 \in \Gamma(Var)$. 
\end{corollary}

\begin{corollary}\label{c: variety M_X}
The variety $\mathsf{M}$ generated by the class $\mathsf{URL}_3$, of residuated lattices on an $\mathbf{M}_X$, is axiomatized relative to $\mathsf{SRL}$ by : $1=$
$$\gamma_1((x_1 \vee x_2) \backslash x_1) \vee
\gamma_2((x_1 \vee x_2 \vee x_3) \backslash (x_1 \vee x_2)) \vee
\gamma_3((x_1 \vee x_2 \vee x_3 \vee x_4) \backslash (x_1 \vee x_2 \vee x_3))$$
where $\gamma_1, \gamma_2, \gamma_3 \in \Gamma(Var)$. 
\end{corollary}

We denote by $\mathsf{bM}$ the corresponding variety of bounded residuated lattices.
Also, we can characterize the finitely subdirectly irreducible algebras in these varieties.

\begin{theorem}\label{semiunilinear_FSI}
    The finitely subdirectly irreducible (FSI) semiunilinear residuated lattices are precisely the unilinear residuated lattices: $\mathsf{SRL}_{FSI} = \mathsf{URL}$.
    More generally, if $\phi$ is a set of positive universal sentences, then the FSIs in  $\mathsf{SRL} \cap \mathsf{V}_\phi$ are precisely the unilinear residuated lattices that satisfy $\phi$.
\end{theorem}

\begin{proof}
It follows from the proof of Theorem 9.73(2) of \cite{galatos2004equational} that an FSI algebra satisfies the unilinearity condition iff it satiefies the equations of Corollary~\ref{c: SRLax}, i.e., iff it is semiunilinear.
So, the semiunilinear FSIs are actually unilinear. 

Conversely, if an algebra is unilinear, then its negative cone ${\downarrow} 1$ is a chain.
Therefore, the convex normal submonoids of the negative cone are nested and $\{1\}$ cannot be the intersection of two non-trivial convex normal submonoids; see \cite{galatos2007residuated} for the correspondence between congruences and convex normal submonoids of the negative cone of residuated lattices.
Therefore, the trivial congruence is meet-irreducible and the algebra is FSI (and semiunilinear, as it is unilinear).
\end{proof}

\begin{corollary}
Every semiunilinear residuated lattice is a subdirect product of unilinear ones. 
\end{corollary}

\begin{corollary}\label{c: FSI variety M_X}
The subdirectly irreducibles in $\mathsf{M}$ are the same as the finitely subdirectly irreducible in $\mathsf{M}$ and as the simple ones in $\mathsf{M}$ and they are precisely the non-trivial residuated lattices based on $\mathbf{M}_X$, for some $X$.
The same holds for $\mathsf{bM}$.
\end{corollary}

That every subdirectly irreducible in each of the varieties $\mathsf{M}$ and $\mathsf{bM}$ is actually simple follows from the fact that its negative cone has two elements.
Consequently, these varieties are semisimple.
For $\mathsf{bM}$ we can say a bit more.

We define the following terms
$$r(x)=(1 \jn x)(1 \mt x) \mt (1 \jn 1 \rd x)(1 \mt 1 \rd x) \qquad x \lra y= x \ld y \mt y \ld x \mt 1$$ $$t(x,y,z)=r(x \lra y)\cdot z \jn (r(x \lra y) \ld \bot \mt 1 ) \cdot x$$

\begin{lemma}
$\mathsf{bM}$ is a discriminator variety with discriminator term $t$.
\end{lemma}

\begin{proof}
If $\m R \in \mathsf{bM}_{SI}$ then, by Corollary~\ref{c: FSI variety M_X}, $\m R$ is a non-trivial bounded residuated lattice based on $\m M_X$ for some $X$.
Note that if $x$ is incomparable to $1$, then also $1 \rd x$ is incomparable to $1$ or is equal to $\bot$, so $1 \mt x = 1 \mt 1 \rd x = \bot$, hence $r(x) = \bot$.
Also, if $x \in \{\bot, \top\}$, then $\{x, 1 \rd x\}= \{\bot, \top\}$, so $1 \mt x = \bot$ or $1 \mt 1 \rd x = \bot$, hence $r(x)=\bot$.
Finally, since $1 \rd 1 = 1$, we have $r(x) = 1$, if $x = 1$ and $r(x) = \bot$ otherwise.

Note that for all $x, y \in R$, we have $x \lra y \leq 1$, i.e., $x \lra y \in \{\bot, 1\}$.
Moreover, $x \lra y = 1$ iff $1 = x \ld y \mt y \ld x \mt 1$ iff $1 \leq  x \ld y \mt y \ld x$ iff ($1 \leq x \ld y$ and $1 \leq \mt y \ld x$) iff ($x \leq y$ and $y \leq x$) iff $x=y$.
Thus we have $x \lra y = 1$ if $x = y$ and $x \lra y = \bot$ if $x \neq y$.

Therefore, $t(x,y,z) = r(1) \cdot z \jn (r(1) \ld \bot \mt 1) \cdot x = 1 \cdot z \jn (1 \ld \bot \mt 1) \cdot x = z \jn \bot \cdot x = z$, if $x=y$; and $t(x,y,z) = r(\bot) \cdot z \jn (r(\bot) \ld \bot \mt 1) \cdot x = \bot \cdot z \jn (\bot \ld \bot \mt 1) \cdot x = (\top \mt 1) \cdot x = 1 \cdot x = x$, if $x \not = y$.
\end{proof}

\subsection{Including (or not) the bounds in the signature}

Note that when axiomatizing classes of unilinear residuated lattices for which the non-linear members are asked to satisfy a certain positive universal sentence, oftentimes the axiomatization looks nicer in the case where the language includes constants for the bounds.
For example, the class of URLs whose non-linear members satisfy $\top x = x\top$ is axiomatized by the positive universal formula
$$u \leq v \text{ or } v \leq v \text{ or }  x (u \jn v)=(u \jn v)x.$$
For non-linear bURL's this formula is equivalent to
$$x \top =\top x.$$

For the sake of readability, we will allow ourselves to denote the first of these sentences as the more pleasing to the eye:
$$x \overline{\top} = \overline{\top} x.$$
We call a (bounded) unilinear residuated lattice \emph{$\top$-central}, if it satisfies this formula.

More generally, if $\Phi$ is the sentence $\forall \vec{x} \, (\varphi(\vec{x}, \top, \bot))$, where $\varphi$ is in the language of URL's, we denote by $\overline{\Phi}$ the sentence
$$\forall \vec{x} \, (\varphi(\vec{x}, \overline{\top}, \overline{\bot})):= \forall u, \forall v, \forall \vec{x} \, (u \leq v \text{ or } v \leq u \text{ or } \varphi(\vec{x}, u \jn v, u \mt v))$$
where $u, v$ are fresh variables.

Likewise, we call a (bounded) unilinear residuated lattice \emph{$\top$-unital}, if it satisfies the formula
$$x = \overline{\bot} \text{ or } x \overline{\top} = \overline{\top} = \overline{\top} x,$$
since in the non-linear models every non-bottom element acts as a unit for the top.
Note that for non-linear bURLs being $\top$-unital is the same as being rigorously compact. 

\begin{lemma}\label{l: Phi}
Let $\varphi$ be a positive universal formula in the language of URLs, let $\Phi$ be $\forall \vec{x} \, (\varphi(\vec{x}, \top, \bot))$ and let $\overline{\Phi}$ be $\forall \vec{x} \, (\varphi(\vec{x}, \overline{\top}, \overline{\bot}))$.
    \begin{enumerate}
        \item The non-linear bURLs that satisfy ${\Phi}$ are precisely the non-linear bURLs that satisfy $\overline{\Phi}$.
        \item The non-linear URLs that satisfy $\overline{\Phi}$ are precisely the bound-free reducts of the non-linear bURLs that satisfy $\overline{\Phi}$.
        \item The linear (bounded) URLs that satisfy $\overline{\Phi}$ are precisely the (bounded) residuated chains.
    \end{enumerate}
\end{lemma}

\begin{proof}
    (1) If $\m R$ is a non-linear bURL, then it satisfies $\overline{\Phi}$ iff it satisfies it for all incomparable elements $u, v$ (as $\overline{\Phi}$ automatically holds for comparable elements $u, v$) iff it satisfies $\Phi$ (since when $u,v$ are incomparable, we have $u \jn v = \top$ and $u \mt v = \bot$).

    (2) follows from the fact that all non-linear URLs are bounded, say $b$ and $t$ are the bounds, and that for bounded non-linear URL's $\overline{\Phi}$ is equivalent to $\forall \vec{x} \, (\varphi(\vec{x}, t, b))$.

    (3) follows from the fact that $\overline{\Phi}$ holds in all totally ordered algebras.
\end{proof}

We note that there might be linear bURLs that satisfy $\overline{\Phi}$, but fail to satisfy $\Phi$. This happens for example when $\Phi$ is $\top x=x \top$.

\subsection{Proof theory for \texorpdfstring{$\mathsf{SRL}$}{SRL}}

Certain varieties of residuated lattices admit a proof-theoretic analysis, which is often complementary to their algebraic study and which often yields interesting results.
Not all varieties of residuated lattices admit a proof-theoretic calculus, but we show that $\mathsf{SRL}$ does admit a hypersequent calculus.
We present the hypersequent system, but we do not pursue any further applications in this paper. 

As a motivating example, we mention the equational theory of lattices, which is axiomatized by the standard basis of the semilattice and the absorption laws.
New valid equations can be derived from these axioms using the derivational system of equational logic, which includes the rules of reflexivity, symmetry, transitivity, and replacement/congruence. 
This system is not amenable to an inverse proof search analysis as, given an equation $s = t$, to determine if it is derivable in the system one cannot simply go through all applications of these derivational rules that could have the equation as a conclusion and proceed recursively: the transitivity rule $\frac{s = t \;\; t = r}{s = r}$ introduces (read upward) a new term that does not appear in the equation.
Also, using inequational reasoning, where for example $\frac{s \leq t \;\; t \leq r}{s \leq r}$ is used instead and the axioms are replaced by inequational axioms such as $s \leq s \jn t$, does not make the problem go away: simply omitting this transitivity rule from the system changes the set of derivable inequalities.
However, a way to bypass this problem is to replace the lattice axioms by inference rules; for example we replace $s \leq s \jn t$ by the inference rule $\frac{r \leq s}{r \leq s \jn t}$.
The axiom and the rule are equivalent in the presence of transitivity, but the rule has elements of transitivity \emph{injected} in it when compared to the axiom: the rule implies the axiom by instantiation, but the axiom implies the rule only with the help of transitivity.
Moreover, the new rule does not suffer from the problem of transitivity as all terms in the numerator are already contained in the denominator; so it is safe to replace the axiom by the rule.
There is a way to inject transitivity into all the axioms, converting them to innocent inference rules, such that in the new system the transitivity rule itself becomes completely redundant.
The resulting system can be used to show the decidability of lattice equations. 

A similar approach works for certain subvarieties of residuated lattices; the axioms in the subvariety may or may not be amenable to injecting transitivity to them.
Also, since there are more operations than in lattices, the above inequalities have to be replaced by \emph{sequents}.
These are expressions of the form $s_1, s_2, \ldots, s_n \Rightarrow s_0$, where the $s_i$'s are residuated-lattice terms, and their interpretation is given by $s_1 \cdot s_2 \cdots s_n \leq s_0$.
The transitivity rule itself at the level of sequents takes the form or a rule called (cut) and the goal is cut-elimination, in the same spirit as above, for lattices; we often write $\Gamma \Rightarrow \Pi$ for sequents, where $\Gamma$ is a sequence of formulas and $\Pi$ is a single formula.
The corresponding derivational systems/calculi define different types of \emph{substructural logics} and varieties of residuated lattices serve as algebraic semantics for them; see \cite{galatos2007residuated}. 

The variety of all residuated lattices admits a sequent derivation system, which leads to the decidability of the equational theory of residuated lattices, among other things.
The variety of semilinear residuated lattices (generated by residuated chains) however, provably does not admit a sequent calculus, due to the shape of its axioms.
It does, however, admit a hypersequent calculus.
Hypersequents are more complex syntactic objects of the form $\Gamma_1 \Rightarrow \Pi_1 \mid \Gamma_2 \Rightarrow \Pi_2  \mid \cdots \mid \Gamma_m \Rightarrow \Pi_m$, i.e., they are multisets of sequents. We denote by $\m {HRL}$ the basis hypersequent system for the variety of residuated lattices; additional inference rules can be added in order to obtain systems for subvarieties. 

We follow \cite{ciabattoni2017algebraic}, which describes the process of injecting transitivity into hypersequents, and we obtain a hypersequent system for the variety $\mathsf{SRL}$ that admits cut elimination.
We start with the axioms of $\mathsf{URL}$, the positive universal class that generates $\mathsf{SRL}$.

First we convert the first axiom $\forall x, y, z (x \leq y \text{ or } y \leq x \text{ or } z \leq (x \vee y))$ to the equivalent form $\forall x, y, z, t_1, t_2, t_3, s_1, s_2, s_3$
    \begin{equation*}
        \begin{split}
        t_1 \leq x \text{ and } y \leq s_1 \text{ and } t_2 \leq y \text{ and } & x \leq s_2 \text{ and } t_3 \leq z \text{ and } (x \vee y) \leq s_3\\
        & \Rightarrow t_1 \leq s_1 \text{ or } t_2 \leq s_2 \text{ or } t_3 \leq s_3    
        \end{split}
    \end{equation*}
by injecting some transitivity.
This then allows to remove the $\jn$ from the axiom, by rewriting it as $\forall x, y, z, t_1, t_2, t_3, s_1, s_2, s_3$
    \begin{equation*}
        \begin{split}
        t_1 \leq x \text{ and } y \leq s_1 \text{ and } t_2 \leq y \text{ and } x \leq s_2 & \text{ and } t_3 \leq z \text{ and } x \leq s_3 \text{ and } y \leq s_3\\
        & \Rightarrow t_1 \leq s_1 \text{ or } t_2 \leq s_2 \text{ or } t_3 \leq s_3    
        \end{split}
    \end{equation*}
In the terminology of \cite{ciabattoni2017algebraic}, the clause is \emph{linear} and \emph{exclusive}, so we eliminate the redundant variables in the premise (noting that $z$ appears only on the right side of inequations, while $x$ and $y$ appear on both sides): we apply transitivity closure and removal of variables in the premise of the clause.
The procedure yields the equivalent clause $\forall t_1, t_2, t_3, s_1, s_2, s_3$
    \begin{equation*}
        \begin{split}
            t_1 \leq s_2 \text{ and } t_1 \leq s_3 \text{ and } t_2 \leq s_1 & \text{ and } t_2 \leq s_3\\
            & \Rightarrow t_1 \leq s_1 \text{ or } t_2 \leq s_2 \text{ or } t_3 \leq s_3
        \end{split}
    \end{equation*}
We now instantiate $s_j$ by $c\ld p_j \rd d$ and use residuation to rewrite $t_i \leq s_j$ as $t_i \leq c\ld p_j \rd d$ and as $ct_id \leq p_j$.
This results in the equivalent clause
 $\forall t_1, t_2, t_3, c, p_1, p_2, p_3, d$
    \begin{equation*}
        \begin{split}
            ct_1d \leq p_2 \text{ and } ct_1d \leq p_3 & \text{ and } ct_2d \leq p_1 \text{ and } ct_2d \leq p_3\\
            & \Rightarrow ct_1d \leq p_1 \text{ or } ct_2d \leq p_2 \text{ or } ct_3d \leq p_3
        \end{split}
    \end{equation*}

Converting the clause to the corresponding hypersequent rule we get 

{\small $$\infer%[\text{(\tiny{URL1})}] 
{\Xi \mid \Gamma, \Sigma_1, \Delta \Rightarrow \Pi_1 \mid \Gamma, \Sigma_2, \Delta \Rightarrow \Pi_2 \mid \Gamma, \Sigma_3, \Delta \Rightarrow \Pi_3}
{\Xi \mid \Gamma, \Sigma_1, \Delta \Rightarrow \Pi_2 \quad \Xi \mid \Gamma, \Sigma_1, \Delta \Rightarrow \Pi_3 \quad \Xi \mid \Gamma, \Sigma_2, \Delta \Rightarrow \Pi_1 \quad \Xi \mid \Gamma, \Sigma_2, \Delta \Rightarrow \Pi_3}$$}

Likewise the second axiom of unilinearity gives the hypersequent rule

{\small
\[
\infer
%[\text{(\tiny{URL2})}]
{\Xi \mid \Gamma, \Sigma_1, \Delta \Rightarrow \Pi_1 \mid \Gamma, \Sigma_2, \Delta \Rightarrow \Pi_2 \mid \Gamma, \Sigma_3, \Delta \Rightarrow \Pi_3}
{\Xi \mid \Gamma, \Sigma_2, \Delta \Rightarrow \Pi_1 \quad \Xi \mid \Gamma, \Sigma_3, \Delta \Rightarrow \Pi_1 \quad \Xi \mid \Gamma, \Sigma_1, \Delta \Rightarrow \Pi_2 \quad \Xi \mid \Gamma, \Sigma_3, \Delta \Rightarrow \Pi_2}
\]
}
We refer to these hypersequent rules as \textup{(URL1)} and \textup{(URL2)}, respectively.

\begin{corollary}
The extension of $\m{HFL}$ with the rules \textup{(URL1)} and \textup{(URL2)} provides a cut-free hypersequent calculus for the variety $\mathsf{SRL}$ by \cite{ciabattoni2017algebraic}.
\end{corollary}

It is notable, that even though $\mathsf{SRL}$ has an infinite equational axiomatization involving iterated conjugates, there are only two inference rules needed for the hypersequent calculus.
This is because hypersequent calculi have the ability to go directly to the level of (finitely) subdirectly irreducibles ($\mathsf{SRL}_{FSI} = \mathsf{URL}$ in this case) and read off the axiomatization from there.

%%%%%%%%%%%%%%%%%%%%%%%%%%%%%%%%%%%%%%%%%%%%%%%%%%%%%%%%%%%%%%%%%%%%%%
\section{Continnum-many subvarieties of \texorpdfstring{$\mathsf{M}$}{M}}\label{s: continuum}

Even though we have a fairly good understanding of the residuated lattices based on $\mathbf{M}_X$, where $X$ is a set, we now show that there are continuum-many subvarieties of $\mathsf{M}$.
More precisely, we will prove that the variety $\mathsf{M_G}$ generated by all the residuated lattices of the form $\mathbf{M}_{\mathbf{G}}$, where $\mathbf{G}$ is an (abelian) group, has continuum-many subvarieties.
We start with an equational basis for $\mathsf{M_G}$.

\begin{prop}\label{p: M_G axioms}
The variety $\mathsf{M_G}$ is axiomatized by the equations
$1 = \gamma_1(u \ld v) \jn \gamma_2(v \ld u) \jn \gamma_3(x \ld (u \mt v)) \jn \gamma_4((u \jn v) \ld x) \jn \gamma_5(x(x\ld 1))$, where $\gamma_1$,$\gamma_2$, $\gamma_3$, $\gamma_4$, $\gamma_5 \in \Gamma(Var)$.
\end{prop}

\begin{proof} 
The formula $x = \overline{\bot} \text{ or }x = \overline{\top} \text{ or } x(x\ld 1)=1$
axiomatizes the FSIs in the variety, so the result follows by Theorem~\ref{semiunilinear_FSI}.
\end{proof}

It is known that there are continuum-many varieties of groups (for example, see \cite{vaughan1970uncountably}) and we can use this fact to show that there is a continuum of subvarieties of $\mathsf{M}_{\mathbf{G}}$, as follows.
Starting with two varieties $\mathcal{V}_1\not = \mathcal{V}_2$ of groups, we can consider the free groups $\m F_1$ and $\m F_2$ on countably many generators in these varieties; hence we have $\mathsf{V}(\m F_1) = \mathcal{V}_1 \not = \mathcal{V}_2 = \mathsf{V}(\m F_2)$.
Then, it is possible to show that $\mathsf{V}(\m M_{\m F_1}) \not = \mathsf{V}(\m M_{ \m F_2})$.

It is also well known that there are only countably-many varieties of abelian groups.
However, we are still able to show that the variety $\mathsf{CM_G}$ of the commutative algebras in $\mathsf{M}_{\mathbf{G}}$ also has continuum-many subvarieties.
Actually, we give a full description of the subvariety lattice of $\mathsf{CM_G}$.

\medskip

We consider the direct power $\bb{N}^\omega$ of countably many copies of the chain $(\mathbb{N}, \leq)$ and its subset $I$ of (not necessarily strictly) decreasing sequences that are eventually zero, such as $(4, 2, 1, 1, 0, 0, \dots)$, $(3, 2, 1, 1, 1, 0, 0, \dots)$ etc.
We will also denote these sequences by $(4, 2, 1, 1)$ and $(3, 2, 1, 1, 1)$, respectively.
It is easy to see that $I$ defines a sublattice $\m{I}$ of the direct product.
We also consider the subset $I^{\oplus \omega}$ of the direct product $\m I^\omega$ of all sequences of elements of $I$ that are eventually the zero sequence.
It is easy to see that this defines a sublattice $\m{I}^{\oplus \omega}$ of the direct product $\m I^\omega$; it makes sense to call $\m{I}^{\oplus \omega}$ the \emph{direct sum} of $\omega$ copies of $\m I$.
We use commas to separate the numbers in each sequence in $I$, but we use semicolons to separate the sequences in each element of $I^{\oplus \omega}$; this allows for dropping parenthesis, if desired.
Therefore, $(2, 1; 3, 1, 1; 0; 2, 1, 1; 0; \dots)$ is an example of an element of $I^{\oplus \omega}$.

Now let $\m{P} = \m 2 \times \m I^{\oplus \omega}$, where $\m 2$ is the two-element lattice on $\{0,1\}$.
For $a \in P$, we define $\expo(a)$ to be the maximum number appearing in $a$; e.g.,
\[
\expo(1; 2, 1; 3, 1, 1; 0; 2, 1, 1; 0; \dots) = 3 \text{ and } \expo(0; 1, 1, 1; 4, 1; 3, 2; 0; \dots) = 4.
\]
Also, for $a \in P$ we write $a = (a_0; a_1; a_2; \ldots)$, where $a_0 \in \{0, 1\}$ and $a_n \in I$, for $n > 0$; we define $\primes(a) = \{n \in \bb{N}: a_n \neq \overline{0}\}$.
For $T \subseteq P$, we define $\expo(T) = \{\expo(a): a \in T\}$ and $\primes(T) = \bigcup \{\primes(a): a \in T\}$.

A downset $D$ of $\m{P}$ is said to be \emph{$\bb{Z}$-closed} if for all $a \in P$, 
\begin{center}
$\expo(D \cap {\uparrow} a)$ or $\primes(D \cap {\uparrow} a)$ is unbounded implies $a \jn (1; 0; 0; \ldots) \in D$.
\end{center}
 
For example, for $a = (0;1;0;0;0;...)$, this condition has the following consequences:
\begin{center}
$(0;1;1;0;0;\ldots), (0;1;2;0;0;\ldots), (0;1;3;0;0;\ldots), \ldots \in D$

or

$(0;1,1;0;0;\ldots), (0;2,1;0;0;\ldots), (0;3,1;0;0;\ldots), \ldots \in D$
\end{center}
implies $(1;1;0;0;0;\ldots) \in D$, because $\expo(D \cap {\uparrow} a)$ is unbounded.
Also,
\begin{center}
$(0;1;1;0;0;\ldots), (0;1;0;1;0;\ldots), (0;1;0;0;1;\ldots), \ldots \in D$
\end{center}
implies $(1;1;0;0;0;\ldots) \in D$, because $\primes(D \cap {\uparrow} a)$ is unbounded.
However, 
\begin{center}
$(0;1;1;0;0;\ldots), (0;1;1,1;0;0;\ldots), (0;1;1,1,1;0;0;\ldots), \ldots \in D$
\end{center}
does not imply $(1;1;0;0;0;\ldots) \in D$.

We denote the lattice of all $\bb{Z}$-closed downsets of $\m{P}$ by $\mathcal{O}_{\bb{Z}}(\m{P})$.

\begin{theorem}\label{t: uncountable varieties}\
The subvariety lattice of $\mathsf{CM_G}$ is isomorphic to $\mathcal{O}_{\mathbb{Z}}(\m P)$.
\end{theorem}

\begin{proof}
Recall that a class of algebras is closed under $\mathsf{HSP_U}$ iff it is axiomatizable by positive universal sentences.
In other words, $\mathsf{HSP_U}$-classes coincide with positive universal classes.

Let $\mathcal{F}$ be a congruence-distributive variety such that $\mathcal{F}_{FSI}$ is a positive universal class.
We claim that the subvarieties of $\mathcal{F}$ are in bijective correspondence with $\mathsf{HSP_U}$-subclasses of $\mathcal{F}_{FSI}$, where the correspondence is given by 
$\mathcal{V} \mapsto \mathcal{V}_{FSI}$ and $\mathcal{K} \mapsto \mathsf{HSP}(\mathcal{K})$; furthermore, it is clear that this correspondence preserves and reflects the inclusion order.
Indeed, $\mathcal{V}_{FSI} = \mathcal{V} \cap \mathcal{F}_{FSI}$, so $\mathcal{V}_{FSI}$ is axiomatized by positive universal sentences and the forward map of the correspodence is well defined.
To show that the two maps are inverses of each other note that $\mathsf{HSP}(\mathcal{V}_{FSI}) \subseteq \mathcal{V} \subseteq \mathsf{SP}(\mathcal{V}_{SI}) \subseteq \mathsf{HSP}(\mathcal{V}_{FSI})$ and by J\' onsson's Lemma $\mathcal{K} = \mathcal{K}_{FSI} \subseteq \mathsf{HSP}(\mathcal{K})_{FSI} \subseteq \mathsf{HSP_U}(\mathcal{K}) = \mathcal{K}$.

Note that residuated lattices form a congruence distributive variety by \cite{galatos2007residuated} and, by Theorem~\ref{semiunilinear_FSI} and Proposition~\ref{p: M_G axioms}, $(\mathsf{CM_G})_{FSI}= \mathsf{CM_G} \cap \mathsf{SRL}_{FSI}$ is axiomatized by positive universal sentences.
So, by the preceding paragraph, the lattice of subvarieties of $\mathsf{CM_G}$ is isomorphic to the lattice of $\mathsf{HSP_U}$-classes of FSIs in $\mathsf{CM_G}$, which by Theorem~\ref{semiunilinear_FSI} and Proposition~\ref{p: M_G axioms} are $\mathsf{HSP_U}$-classes of algebras of the form $\m{M}_{\m{G}}$, where $\m{G}$ is an abelian group. 

Further note that $\mathsf{H}$ can be replaced by $\mathsf{I}$.
Indeed, every ultrapower of algebras of the form $\m{M}_{\m{G}}$, where $\m{G}$ is an abelian group, is also an algebra of the same form (it satisfies the same first-order sentences, hence also all positive universal sentences).
Also, subalgebras are also of the same form (where we also include the trivial algebra). 
Finally, since every algebra of this form is simple (since their lattice reduct is simple), $\mathsf{H}$ does not contribute any new algebras.
So we are interested in $\mathsf{ISP_U}$-classes of algebras of the form $\m{M}_{\m{G}}$, where $\m{G}$ is an abelian group.

We now prove that such classes are in bijective correspondence with $\mathsf{ISP_U}$-classes of abelian groups, by showing that for every class $\mathcal{K}$ of abelian groups, we have $\mathsf{ISP_U}(\{\m{M}_{\m{H}}: \m H \in \mathcal{K}\}) = \mathsf{I}\{\m{M}_{\m{G}}: \m{G} \in \mathsf{SP_U}(\mathcal{K})\}$ and thus this class can be associated with $\mathsf{ISP_U}(\mathcal{K})$; clearly this corresponcence preserves and reflects the order.

First we show $\mathsf{IP_U}(\{\m{M}_{\m{H}}: \m H \in \mathcal{K}\}) = \mathsf{I}\{\m{M}_{\m{G}}:  \m{G} \in \mathsf{IP_U}(\mathcal{K})\}$.
For a residuated lattice $\m R$, if $\m R \in \mathsf{IP_U}(\{\m{M}_{\m{H}}: \m H \in \mathcal{K}\})$, then $\m R$ satisfies all first-order sentences that hold in the $\mathbf{M_H}$'s, where $\m H \in \mathcal{K}$.
In particular, $\m R$ is commutative, unilinear, has height at most $3$, and all of its non-bound elements are invertible, closed under multiplication and serve as units for the top.
Therefore, $\m R$ is isomorphic to $\mathbf{M_G}$ for some abelian group $\m G$.
Also, clearly, all algebras in $\mathsf{P_U}(\mathcal{K})$ are abelian groups.
Therefore the classes on both sides of the equation contain only algebras isomorphic to $\mathbf{M_G}$ for some abelian group $\m G$, and it is enough to focus on such algebras: we show that for every abelian group $\m G$, $\mathbf{M_G} \in \mathsf{IP_U}(\{\m{M}_{\m{H}}: \m H \in \mathcal{K}\})$ iff  $\m{G} \in \mathsf{IP_U}(\mathcal{K})$; we will identify the bounds in all algebras to omit $\mathsf{I}$.

If $\mathbf{M_G} \in \mathsf{P_U}(\{\m{M}_{\m{H}}: \m H \in \mathcal{K}\})$, there exists an index set $I$, an ultrafilter $U$ on $I$ and $\m H_i \in \mathcal{K}$, $i \in I$, such that
$\mathbf{M_G} = \prod \mathbf{M_{H_i}} / U$.
So, for every $g \in G$ there exists $x_g \in \prod \mathbf{M_{H_i}}$ such that $g = [x_g]$, the equivalence class of $x_g$.
We will use $\overline{\top}$ and $\overline{\bot}$ to denote the tuples $(\top)_{i \in I}$ and $(\bot)_{i \in I}$ in $ \prod \mathbf{M_{H_i}}$ respectively.
Then for all $g \in G$, we have $g \neq [\overline{\top}]$ and $g \neq [\overline{\bot}]$, since $g$ is invertible while $[\overline{\top}]$ and $[\overline{\bot}]$ are idempotents different than the identity.
So we know $\{i \in I: x_g(i) \neq \top\} \in U$ and $\{i \in I: x_g(i) \neq \bot\} \in U$, hence $\{i \in I: x_g(i) \in H\} = \{i \in I: x_g(i) \neq \top\} \cap \{i \in I: x_g(i) \neq \bot\} \in U$.
Now define a tuple $x$ in $ \prod \mathbf{{H_i}}$ by $x(i) = x_g(i)$, if $x_g(i) \in H_i$, and $x(i) = 1$ otherwise.
Then we have $g = [x_g] = [x] \in \prod H_i / U$, so $\mathbf{G} \in \mathsf{P_U}(\mathcal{K})$.

If $\m{G} \in \mathsf{P_U}(\mathcal{K})$, then there exists an index set $I$, an ultrafilter $U$ on $I$ and $\m H_i \in \mathcal{K}$, $i \in I$, such that $\mathbf{G} = \prod \mathbf{H_i} / U$.
Using the same index set $I$ and ultrafilter $U$ on $I$, we know $\prod \mathbf{M_{H_i}} / U$ is also of the form $\m{M_K}$, where $\m K$ is an abelian group.
Since $[\overline{\top}] \vee [x] = [\overline{\top} \vee x] = [\overline{\top}]$ and $[\overline{\bot}] \wedge [x] = [\overline{\bot} \wedge x] = [\overline{\bot}]$, we get $[\overline{\top_{\m{M_{H_i}}}}] = \top_{\prod \mathbf{M_{H_i}} / U}$ and $[\overline{\bot_{\m{M_{H_i}}}}] = \bot_{\prod \mathbf{M_{H_i}} / U}$.
For $[x] \in K$, we have $[x] \neq [\overline{\top_{\m{M_{H_i}}}}]$ and $[x] \neq [\overline{\bot_{\m{M_{H_i}}}}]$.
So $\{i \in I: x(i) \neq \top_{\m{M_{H_i}}}\} \in U$ and $\{i \in I: x(i) \neq \bot_{\m{M_{H_i}}}\} \in U$, hence $\{i \in I: x(i) \in {H_i}\} = \{i \in I: x(i) \neq \top_{\m{M_{H_i}}}\} \cap \{i \in I: x(i) \neq \bot_{\m{M_{H_i}}}\} \in U$; so $[x] \in \prod {H}_i / U = \m{G}$ and $K \subseteq G$.
Conversely, if $[x] \in \prod {H}_i / U = \m{G}$ then $[x] \in K$, so $G \subseteq K$.
Therefore $\m{M_G} \in \mathsf{P_U}(\{\m{M}_{\m{H}}: \m H \in \mathcal{K}\})$.

Again note that to show $\mathsf{S}(\m{M}_{\m{H}}) = \{\m{M}_{\m{G}}: \m{G} \in \mathsf{S}(\m{H})\}$ it is enough to focus on algebras of the form $\m M_{\m G}$, where $\m G$ is an abelian group.
If $\mathbf{M_G} \in \mathsf{S}(\mathbf{M_H})$, then for all $x, y \in G$, we have $x \cdot_{\mathbf{G}} y = x \cdot_{\mathbf{M_G}} y = x \cdot_{\mathbf{M_H}} y = x \cdot_{\mathbf{H}} y$ and $x^{-1_{\mathbf{G}}} = x \ld_{\mathbf{M_G}} 1 = x \ld_{\mathbf{M_H}} 1 = x^{-1_{\mathbf{H}}}$; so $\mathbf{G} \in \mathsf{S}(\mathbf{H})$.
Conversely, if $\m{G} \in \mathsf{S}(\m{H})$, then for all $x, y \in M_G \setminus \{\bot, \top\}$ we have
$x \cdot_{\m M_G} y = x \cdot_{\m G} y = x \cdot_{\m H} y = x \cdot_{\m M_H} y$, $x \backslash_{\m M_G} y = x^{-1_{\m G}} \cdot_{\m G} y = x^{-1_{\m H}} \cdot_{\m H} y = x \backslash_{\m M_H} y$ and $y /_{\m M_G} x = y \cdot_{m_G} x^{-1_{\m G}} = y \cdot_{m_H} x^{-1_{\m H}} = y /_{\m M_H} x$.
Also, since $\m M_G$ is rigorously compact, the operations on $\m G$ and $\m H$ also agree if one of $x$, $y$ is in $\{\bot, \top\}$.
So $\m{M_G} \in \mathsf{S}(\m{M_H})$.

Actually, given that every algebra is an ultraproduct of its finitely generated subalgebras, $\mathsf{ISP_U}$-classes of abelian groups are fully determined by their intersection with the class of finitely generated abelian groups.
Therefore, we are interested only in such intersections; clearly this corresponcence preserves and reflects the order.

By the fundamental theorem of finitely generated abelian groups we know that every finitely generated abelian group is isomorphic to exactly one group of the form
\[
\mathbb{Z}^m \times (\mathbb{Z}_{p_1^{n_{1, 1}}} \times \cdots \times \mathbb{Z}_{p_1^{n_{1,m_1}}}) \times \cdots \times (\mathbb{Z}_{p_k^{n_{k, 1}}} \times \cdots \times \mathbb{Z}_{p_k^{n_{k,m_k}}})  
\]
for some $m, k, m_1, \dots, m_k,n_{i,j} \in \mathbb{N}$, where $n_{i,j} \geq n_{i,j+1}$ for all suitable $i, j$, and $p_1 < p_2 < \dots < p_k < \ldots $ is the listing of all primes.
We denote by $\mathcal{FA}$ the set of all groups of this form; also by $f\mathcal{A}$ we denote all the finite algebras in $\mathcal{FA}$ (i.e., where $m=0$).

Since $\mathcal{FA}$ is a full set of representatives of the isomorphism classes of finitely generated abelian groups, instead of considering intersections of $\mathsf{ISP_U}$-classes of abelian groups with the class of finitely generated abelian groups, we can instead focus on intersections of $\mathsf{ISP_U}$-classes of abelian groups with $\mathcal{FA}$.
In other words, we have established that the subvariety lattice of $\mathsf{CM_G}$ is isomorphic to $\{\mathcal{K} \cap \mathcal{FA}: \mathcal{K} \text{ is an } \mathsf{ISP_U} \text{-class of abelian groups}\}$, where the order is given by: $\mathcal{K} \cap \mathcal{FA} \leq \mathcal{L} \cap \mathcal{FA}$ iff
$\mathsf{ISP_U}(\mathcal{K} \cap \mathcal{FA}) \subseteq \mathsf{ISP_U}(\mathcal{L} \cap \mathcal{FA})$.
In the following, we will write $\mathcal{K}_{\mathcal{FA}}$ for $\mathcal{K} \cap \mathcal{FA}$.

To the abelian group displayed above, we associate the sequence
$$(m; (n_{1, 1}, \ldots, n_{1, m_1}, 0, \ldots); \ldots; (n_{k, 1}, \ldots, n_{k, m_k}, 0, \ldots); (0, \ldots ); \ldots )$$
which is an element of the lattice $\mathbb{N} \times \m I^{\oplus \omega}$. 
Also, note that the bijective correspondence from $\mathcal{FA}$ to $\mathbb{N} \times I^{\oplus \omega}$ is actually a lattice isomorphism between $\mathbb{N} \times \m I^{\oplus \omega}$ and $\mathcal{FA}$ under the order given by: $\m G \leq_{\mathcal{FA}} \m H$ iff $\m G \in \mathsf{IS}(\m H)$.

Now, sets of the form $\mathcal{K}_{\mathcal{FA}}$, where $\mathcal{K}$ is an $\mathsf{ISP_U}$-class of abelian groups, are of course downsets of $\mathcal{FA}$, but unfortunately not all downsets of $\mathcal{FA}$ are of this form.
For example, note that for $r, s \in \mathbb{Z}^+$, $\m G \in f\mathcal{A}$ and $\mathcal{K}$ an $\mathsf{ISP_U}$-class of abelian groups, we have:
$\m G \times \mathbb{Z}^r \in \mathcal{K}$ iff $\m G \times \mathbb{Z}^s \in \mathcal{K}$.
(So, for example ${\downarrow} \{\mathbb{Z}^2\} = \{\{1\}, \mathbb{Z}^2, \mathbb{Z}\}$ is a downset of $\mathcal{FA}$ that is not of the form $\mathcal{K}_{\mathcal{FA}}$.)

To prove this, it suffices to prove: if $\m G \times \bb{Z} \in \mathcal{K}$ then $\m G \times \bb{Z}^t \in \mathcal{K}$ for all $t \in \bb{Z}^+$.
Let $U$ be a non-principal ultrafilter on $\bb{N}$ and consider the elements $a = [\overline{1}]_U$ and $b = [(2, 2^2, 2^3, \dots)]_U$ of $\bb{Z}^{\bb{N}} / U$; each has infinite order.
Note that for all $m, n \in \bb{N}$, the set $\{i \in \bb{N}: m \cdot 1 = n \cdot 2^i\}$ contains at most one element.
Since  $U$ is not principal, we get $\{i \in \bb{N}: m \cdot 1 = n \cdot 2^i\} \not \in U$, so $ma \neq nb$.
Thus $\langle a, b \rangle \cong \bb{Z} \times \bb{Z}$ and $\bb{Z} \times \bb{Z} \in \mathsf{P_U}(\bb{Z})$.
Similarly, to show $\bb{Z}^t \in \mathsf{P_U}(\bb{Z})$, it suffices to take $a_{p_1} = [(p_1, p_1^2, p_1^3, \dots)]_U$, $a_{p_2} = [(p_2, p_2^2, p_2^3, \dots)]_U$, \dots, $a_{p_t} = [(p_t, p_t^2, p_t^3, \dots)]_U$, where $p_1, p_2, \ldots, p_t$ are distinct primes, and we have $\langle a_{p_1}, \dots, a_{p_t} \rangle \cong \bb{Z}^t$.
More generally, we can show $\{\m G \times \bb{Z}^t: t \in \bb{Z}^+\} \subseteq \mathsf{P_U}(\m G \times \bb{Z})$ for any $\m G \in f\mathcal{A}$.

For this reason, it makes sense to identify $\m G \times \mathbb{Z}^r$ and $\m G \times \mathbb{Z}^s$ whenever $r$ and $s$ are both non-zero.
This can be done by considering the subset $\mathcal{FA}' = f\mathcal{A} \cup \{\bb{Z} \times \m G: \m G \in f\mathcal{A}\}$ of $\mathcal{FA}$.
The set $\mathcal{FA}'$ also forms a lattice (actually a sublattice of $\mathcal{FA}$) isomorphic to $\m P = \m 2 \times \m I^{\oplus \omega}$.
Therefore, moving through the isomorphism, we can apply the definitions of $\expo$ and $\primes$ also to downsets of $\mathcal{FA}'$.
To be more specific, a downset $D$ of $\mathcal{FA}'$ is $\bb{Z}$-closed if for all $\m G \in f\mathcal{A}$, $\expo(D \cap {\uparrow} \m G)$ or $\primes(D \cap {\uparrow} \m G)$ being unbounded implies that $\bb{Z} \times \m G \in D$.
Also, by the fact established in the last paragraph we have a lattice isomorphism between $\{\mathcal{K}_{\mathcal{FA}}: \mathcal{K} \text{ is a } \mathsf{ISP_U} \text{-class}\}$ and $\{\mathcal{K}_{\mathcal{FA}'}: \mathcal{K} \text{ is a } \mathsf{ISP_U} \text{-class}\}$, where $\mathcal{K}_{\mathcal{FA}'} = \mathcal{K} \cap \mathcal{FA}'$.

Clearly, if $\mathcal{K}$ is an $\mathsf{ISP_U}$-class of abelian groups, then $\mathcal{K}_{\mathcal{FA}'}$ is a downset of $\mathcal{FA}'$.
Unfortunately, still not every downset of $\mathcal{FA}'$ is of this form.
For example, $\{\bb{Z}_p: p \text{ is prime}\}$ is a downset of $\mathcal{FA}'$, but since $\bb{Z} \in \mathsf{P_U}(\{\bb{Z}_p: p \text{ is prime}\})$, $\{\bb{Z}_p: p \text{ is prime}\}$ is not of the form $\mathcal{K}_{\mathcal{FA}'}$.
In the following we show that $\{\mathcal{K}_{\mathcal{FA}'}: \mathcal{K} \text{ is an } \mathsf{ISP_U} \text{-class}\}$ is equal to the lattice of $\bb{Z}$-closed downsets of $\mathcal{FA}'$.

First we note that for $X \subseteq P$, we have that $\expo(X)$ and $\primes(X)$ are bounded iff there exist $K, N \in \mathbb{N}$ such that for all $a \in X$, $k > K$, $n, m \in \mathbb{N}$, we have $a_k = \overline{0}$ and $a_{n,m}\leq N$.
Therefore, for $X \subseteq \mathcal{FA}'$, we have that $\expo(X)$ and $\primes(X)$ are bounded iff there exist $K, N \in \mathbb{N}$ such that the cyclic groups in the decomposition of groups in $X$ are among the $\m Z_{p_k^n}$, where $k \leq K$ and $n \leq N$.
This is in turn equivalent to asking that there is $M \in \mathbb{N}$ such that all elements in all the groups in $X$ have order at most $M$ (by taking $M = (p_1 \cdots p_K)^N$).

Now, for an $\mathsf{ISP_U}$-class $\mathcal{K}$ of abelian groups, $\mathcal{K}_{\mathcal{FA}'}$ is a downset of $\mathcal{FA}'$.
To show that it is $\mathbb{Z}$-closed, let $\m G \in f\mathcal{A}$.
If one of $\expo(\mathcal{K}_{\mathcal{FA}'} \cap \, {\uparrow} \m G)$, $\primes(\mathcal{K}_{\mathcal{FA}'} \cap \, {\uparrow} \m G)$ is unbounded, there is no uniform bound in the order of the elements in the groups from $\mathcal{K}_{\mathcal{FA}'}$; so, there is an infinite subset $\{\m H_n : n \in \mathbb{N}\}$ of $\mathcal{K}_{\mathcal{FA}'} \cap {\uparrow} \m G$ such that $\m H_n$ contains an element of order greater than $n$, say $h_n$.
Therefore, the element $[(h_n)]$ in any  fixed non-principal ultraproduct $\m H$ of $\{\m H_n : n \in \mathbb{N}\}$ has infinite order, and consequently $\m H$ contains a copy of $\mathbb{Z}$.

On the other hand, note that if $\m G = \{g_1, \ldots, g_k\}$, then for every group $\m A$ we have $\m G \in \mathsf{IS}(\m A)$ iff $\m A \vDash \phi_{\m G}$, where $\phi_{\m G}$ encodes the multiplication of $\m G$: 
$\exists x_{g_1}, \dots, x_{g_k} \, 
(\bigwedge\{x_{g_i} \not = x_{g_j}: i \not = j \} 
\wedge 
\bigwedge \{x_{g_i} x_{g_j} = x_{g_i g_j}: 1 \leq i, j \leq n \}
)$.
Since, for all $n$, $\m H_n$ contains a copy of $\m G$, $\m H_n$ satisfies $\phi_{\m G}$; hence $\m H$ also satisfies $\phi_{\m G}$ and $\m H$ contains a subgroup isomorphic to $\m G$.
Therefore, $\bb{Z} \times \m G \in \mathsf{IS}(\m H) \subseteq \mathsf{IS}(\mathcal{K}) = \mathcal{K}$ and so $\bb{Z} \times \m G \in \mathcal{K}_{\mathcal{FA}'}$. 

Conversely, for a $\bb{Z}$-closed downset $D$ of $\mathcal{FA}'$, we define $\mathcal{K}_D = \mathsf{ISP_U}(D)$ and prove that $\mathcal{K}_D \cap \mathcal{FA}' = D$.
Since $D \subseteq \mathcal{K}_D$ and $D \subseteq \mathcal{FA}'$, it suffices to prove $\mathcal{K}_D \cap \mathcal{FA}' \subseteq D$.
If $\mathbb{Z}^m \times \m G \in \mathcal{K}_D \cap \mathcal{FA}'$, where $m \in \{0,1\}$ and $\m G \in f\mathcal{A}$, then a copy of $\mathbb{Z}^m \times \m G$ is contained in the ultraproduct $\prod \m A_i / U$ of some $\{\m A_i: i \in I\} \subseteq D$.
Since $\prod \m A_i / U$ contains a copy of $\m G$, it satisfies the sentence $\phi_{\m G}$, so $I_{\m G} := \{i \in I: \m G \in \mathsf{IS}(\m A_i)\} = \{i \in I: \m A_i \vDash \phi_{\m G}\} \in U$.
If $m = 1$, then $\prod \m A_i / U$ contains a copy of $\mathbb{Z}$, so it has an element of infinite order.
Therefore, there is no $M$ such that $\prod \m A_i / U$ satisfies the sentence $(\forall x)(Mx = 0)$, so there is no $M$ such that $\{\m A_i: i \in I_{\m G}\}$ satisfy the sentence, so there is no uniform bound on the orders of the elements of $\{\m A_i: i \in I_{\m G}\}$; thus $\expo(\{\m A_i: i \in I_G\})$ or $\primes(\{\m A_i: i \in I_{\m G}\})$ is unbounded.
Since, $\expo(\{\m A_i: i \in I_{\m G}\}) \subseteq \expo(D \cap {\uparrow} \m G)$, $\primes(\{\m A_i: i \in I_{\m G}\}) \subseteq \primes(D \cap {\uparrow} \m G)$ and $D$ is a $\bb{Z}$-closed downset, we get  $\mathbb{Z}^m \times \m G= \mathbb{Z}\times \m G \in D$. 
If $m = 0$, then we also have $\mathbb{Z}^m \times \m G = \m G \in D$.

Thus the lattice $\{\mathcal{K}_{\mathcal{FA}'}: \mathcal{K} \text{ is a } \mathsf{ISP_U} \text{-class}\}$ is isomorphic to $\mathcal{O}_{\bb{Z}}(\m P)$, and hence the lattice $\Lambda(\mathsf{CM_G})$ of subvarieties of $\mathsf{CM_G}$ is isomorphic to the lattice $\mathcal{O}_{\bb{Z}}(\m P)$.
\end{proof}

\begin{corollary}
The variety generated by $\{\m M_{\mathbb{Z}_p}: p \text{ is prime}\}$ has continuum-many subvarieties. 
Therefore the subvariety lattices of $\mathsf{M_G}$ and of $\mathsf{M}$ have size continuum.
\end{corollary}

\begin{proof}
For every prime $p$, the variety $\mathsf{V}(\m M_{\mathbb{Z}_p})$ corresponds to the principal downset of the sequence $(0; 0; \ldots; 0; 1; 0; \ldots )$ in $\m P$, where the $1$ is at the position of the prime $p$.
The variety generated by all $\m M_{\mathbb{Z}_p}$'s is the join of all of the $\mathsf{V}(\m M_{\mathbb{Z}_p})$, where $p$ is prime, and corresponds to the $\mathbb{Z}$-closed downset 
\[
\overline{P\mathbb{N}}: = \{(1; 0; 0; \ldots), (0; 1; 0; \dots), \dots, (0; \dots; 1; 0; \dots), \dots\}
\]
in $\m P$.
The $\bb{Z}$-closed subdownsets of $\overline{P\mathbb{N}}$ in the lattice $\mathcal{O}_{\mathbb{Z}}(\m P)$ is clearly isomorphic, as a lattice, to $\mathcal{P}(\mathbb{N})$.
\end{proof}

We denote by $\mathsf{CM}_{\m GZ}$ the variety generated by the algebras in $\mathsf{M}$ that satisfy the formula
\begin{equation}\tag{ZGroup} \label{ZGroup}
    x \overline{\top} = x \text{ or } x(x \ld 1) = 1.
\end{equation}

Let $\m F$ be the poset on $\{0,1,2,3\}$, where $0 < 1, 2, 3$ and $1, 2, 3$ are incomparable.
For a downset $D$ of $\m P \times \m F$ and $i \in F$, we set $D_i = \{a: (a, i) \in D\}$. 
A downset $D$ of $\m P \times \m F$ is called $\mathbb{Z}$-closed if $D_0$, $D_1$, $D_2$ and $D_3$ are $\mathbb{Z}$-closed downsets of $\m P$; we denote by $\mathcal{O}_{\mathbb{Z}}(\m P \times \m F)$ the lattice of all $\mathbb{Z}$-closed downsets of $\m P \times \m F$.

\begin{theorem}\label{t: uncountable varieties +Z}\
The subvariety lattice of $\mathsf{CM_{GZ}}$ is isomorphic to $\mathcal{O}_{\mathbb{Z}}(\m P \times \m F)$.
\end{theorem}

\begin{proof}
By Theorem~\ref{semiunilinear_FSI} and Corollary~\ref{Char_of_RL_on_M_X} the FSI members of $\mathsf{CM_{GZ}}$ are unilinear residuated lattices of the form $\m R$, $\m R+1$, $\m R+2$ or $\m R+3$, where $\m R = \m M_{\m G}$ and $\m G$ is an abelian group, $\m A$ is the $\top$-cancellative monoid on $G \cup \{\top\}$;
$\m R+1 = \m R_{\m A, \m B_1}$, where $\m B_1$ is the $\bot$-semigroup based on $\{\bot, b\}$ given in
Figure~\ref{f:4tables} with $b^2 = \bot$; $\m R+2 = \m R_{\m A, \m B_2}$, where $\m B_2$ is the $\bot$-semigroup based on $\{\bot, b_1, b_2\}$ given in Figure~\ref{f:4tables}; and $\m R+3 = \m R_{\m A, \m B_3}$, where $\m B_3$ is the $\bot$-semigroup based on $\{\bot, b\}$ given in Figure~\ref{f:4tables} with $b^2 = b$; we define $\m R+0 = \m R$.
Note that $\m R$ is a subalgebra of $\m R+i$, for all $i \in \{0, 1, 2, 3\}$.

In the proof of Theorem~\ref{t: uncountable varieties}, we saw that subvarieties of $\mathsf{CM_G}$ are determined by the $\mathbb{Z}$-closed downsets of $\mathcal{FA}'$.
We now sketch how subvarieties of $\mathsf{CM_{GZ}}$ are determined by the $\mathbb{Z}$-closed downsets of the poset $\m M_{\mathcal{FA}'} + \m F := \{\m M_{\m G}+i: \m G \in \mathcal{FA}', i \in F\}$, where the order is given by $\m M_{\m G}+i \leq \m M_{\m H}+j$ iff $\m G \leq_{\mathcal{FA}'} \m H$ and $i \leq_{\m F} j$; this poset is clearly isomorphic to $\m P \times \m F$, so the definition of $\bb{Z}$-closed downsets of $\m P \times \m F$ can be transferred here.
More specifically, a downset $D$ of $\m M_{\mathcal{FA}'} + \m F$ is $\bb{Z}$-closed iff for all $0 \leq i \leq 3$, $D \cap (\m M_{\mathcal{FA}'}+\{i\})$ is isomorphic to a $\bb{Z}$-closed downset of $\mathcal{FA}'$.
    
Every subvariety $\mathcal{V}$ of $\mathsf{CM_{GZ}}$ is determined by its finitely generated FSI algebras.
These are finitely generated algebras of the form $\m R$, $\m R+1$, $\m R+2$ or $\m R+3$, where $\m R \in (\mathsf{CM_G})_{FSI}$, i.e., $\m R = \m M_{\m G}$, and $\m G$ is a finitely generated abelian group.
So, $\mathcal{V}_{FSI}$ is a downset of $\m M_{\mathcal{FA}'}+\m F$.
   
For $0 \leq i \leq 3$, if $\m G \in f\mathcal{A}$ and $\expo(D_i \cap {\uparrow} G)$ or $\primes(D_i \cap {\uparrow} G)$ is unbounded, where $D_i = \{\m K \in \mathcal{FA}': \m M_{\m K}+i \in \mathcal{V}_{FSI}\}$, then by the proof of Theorem~\ref{t: uncountable varieties}, we have $\bb{Z} \times \m G \in D_i$.
So $D_i$ is a $\bb{Z}$-closed downset of $\mathcal{FA}'$ for $0 \leq i \leq 3$ and hence $\mathcal{V}_{FSI}$ is a $\bb{Z}$-closed downset of $M_{\mathcal{FA}'}+\m F$.
    
By Corollary~\ref{Char_of_RL_on_M_X}, for every downset $D$ of $\m M_{\mathcal{FA}'}+\m F$, the ultraproducts of algebras from $D$ are isomorphic to $\m M_{\m G}+i$, for some $0 \leq i \leq 3$.
It can be easily shown that for such ultraproduct $\m M_{\m G}+i$, $\m G$ is an ultraproduct of $\{\m H: i \leq_{\m F} j, \m M_{\m H}+j \in D\}$; since $D$ is a downset, actually $\m G$ is an ultraproduct of $\{\m H: \m M_{\m H} +i \in D\}$.
(Also, conversely, if $\m G$ is an ultraproduct of $\{\m H_j: j \in J\}$ and $i \in F$, then $\m M_{\m G}+i$ is isomorphic to an ultraproduct of algebras in the downset $\{\m M_{\m K_j}+k: j \in J, \m K_j \leq_{\mathcal{FA}'} \m H_j, k \leq_{\m F} i\}$ of $\m M_{\mathcal{FA}'}+\m F$.)
So if $D$ is a $\bb{Z}$-closed, then $\m G \in D_i$; hence $\m M_{\m G}+i \in D$.
Consequently, we have $\mathsf{ISP_U}(D) \cap \m (M_{\mathcal{FA}'}+\m F) = D$, hence the subvariety lattice of $\mathsf{CM_{GZ}}$ is isomorphic to $\mathcal{O}_{\bb{Z}}(\m P \times \m F)$.
\end{proof}

%%%%%%%%%%%%%%%%%%%%%%%%%%%%%%%%%%%%%%%%%%%%%%%%%%%%%%%%%%%%%%%%%%%%%%
\section{The finite embeddability property}\label{s: FEP}

In this section we establish the finite embeddability property for certain subvarieties of $\mathsf{SRL}$.

Recall that a class $\mathcal{K}$ is said to have the \emph{finite embeddability property} (FEP) if for every algebra $\m A \in \mathcal{K}$ and a finite subset $B$ of $A$, there exists a finite algebra $\m C \in \mathcal{K}$ such that the partial subalgebra $\m B$ of $\m A$ induced by $B$ embeds in $\m C$.

For varieties axiomatized by a recursive set of equations, the valid universal sentences form a recursively enumerable set.
Also, if the variety has the FEP, then any universal sentence that is not valid will fail in a finite algebra of the variety.
By enumerating these finite algebras (using the finite axiomatizability of the variety) we can thus enumerate the universal sentences that fail in the variety.
Therefore, recursively axiomatizable varieties with the FEP have a decidable universal theory; moreover, they are generated as universal classes (thus also as quasivarieties and as varieties) by their finite algebras. 

\begin{theorem}
The variety $\mathsf{CM_G}$ has the FEP.  
\end{theorem}

\begin{proof}
To prove this, first we claim that the variety of abelian groups has FEP.
By Theorem~5.1 of \cite{schein1966homomorphisms}, an abelian group is subdirectly irreducible if and only if it is a subgroup of a $p$-cyclic group, i.e., either it is a $p^{\infty}$-group or a cyclic group of order $p^n$, where $p$ is a prime.
So every finitely generated subdirectly irreducible abelian group is finite.
By Corollary 2 in \cite{carlisle1971residual} every finitely generated abelian group is residually finite.
By Theorem 1 in \cite{evans1969some} this is equivalent to having the FEP, so the variety of abelian groups has the FEP.

Note that the above characterization of the finitely generated subdirectly irreducibles does not extend to algebras in $\mathsf{CM_G}$, since the notion of subdirectly irreducible is different.
Nevertheless, we can make use of the FEP for abelian groups. 

It suffices to prove the FEP for the subdirectly irrducible algebras in $\mathsf{CM_G}$.
Let $\m G$ be an abelian group and $\mathbf{B}$ a finite subset of $\mathbf{M_G}$.
Without loss of generality, we can assume $\bot, \top \in B$, where $\top$ and $\bot$ denote the bounds of $\mathbf{M_G}$, so $(B, \wedge, \vee)$ is a sublattice of $\m M_G$.
Then $(B', \cdot, 1)$ is a finite partial subgroup of $\m G$, where $B' = B \setminus \{\top, \bot\}$.
By the FEP for abelian groups, there exists a finite abelian group $\m C'$ such that $(B', \cdot, 1)$ can be embedded into $\m C'$; without loss of generality we assume that $B' \subseteq C'$.

We consider the set $C = C' \cup \{\top, \bot\}$ and define an order keeping the elements of $C'$ incomparable and setting $\bot < x < \top$, for all $x \in C'$.
Also, we extend the multiplication of $\m C'$ by stipulating that $\top$ is absorbing for $C \cup\{\top\}$ and $\bot$ is absorbing for $C'$.
Finally, we define $x \rightarrow y = x^{-1} \cdot y$ for $x \in C'$, $\top \ra u = \bot = v \ra \bot$ for $u \not = \top$ and $v \neq \bot$, and $w \ra \top = \top = \bot \ra w$, for all $w$.

Since $(B, \wedge, \vee)$ is a sublattice of $\m M_G$ and $B' \subseteq C'$, $(B, \wedge, \vee)$ is a sublattice of $(C, \wedge, \vee)$.
For all $x, y \in B'$, if $x \cdot_{\m B} y \in B$, then $x \cdot_{\m B} y = x \cdot_{\m B'} y = x \cdot_{\m C'} y = x \cdot_{\m C} y$, since $\m G$ is closed under multiplication; if $x \ra_{\m B} y \in B$, then $x^{-1_{\m B}} \in B$ and $x \ra_{\m B} y = x^{-1_{\m B}} \cdot_{\m B} y = x^{-1_{\m B'}} \cdot_{\m B'} y = x^{-1_{\m C'}} \cdot_{\m C'} y = x \ra_{\m C} y$, since $\m G$ is also closed under inverses.
Finally, if $x, y \in B$ and $x \in \{\bot, \top\}$ or $y \in \{\bot, \top\}$, then the embedding works since $\bot \ra_{\m {M_G}} a = \top = a \ra_{\m {M_G}} \top$, $a \bot = \bot = \bot a$ for all $a \in M_G$ and $b \ra_{\m {M_G}} \bot = \bot = \top \ra_{\m {M_G}} c$, $b \top = \top = \top b$ for all $b \neq \bot$ and $c \neq \top$.
\end{proof}

\begin{corollary}
    The universal theory of the variety $\mathsf{CM_G}$ is decidable.  
\end{corollary}

We can actually prove the FEP for many more subvarieties of $\mathsf{SRL}$, unrelated to $\mathsf{GM_G}$, using a construction based on residuated frames.

An equation is called \emph{knotted} if it is of the form $x^m \leq x^n$, where $n \not = m$.
Also, we consider the following weak versions of commutativity.
For every $n \in \mathbb{Z}^+$ and non-constant \emph{partition} $a$ of $n+1$ (i.e., $a = (a_0, a_1, \ldots, a_n)$, where $a_0 + a_1 + \dots + a_n = n+1$ and not all $a_i$'s are $1$), we consider the ($n+1$)-variable identity ($a$):
$$xy_1 xy_2 \cdots y_n x = x^{a_0} y_1 x^{a_1} y_2 \cdots y_n x^{a_n}.$$
For example, ($2, 0$) is the identity $xyx = xxy$ and ($2, 0, 1$) is the identity $xyxzx = xxyzx$.
We call all of these identities \emph{weak commutativity} identities.

\begin{theorem}
    If a subvariety of $\mathsf{SRL}$ is axiomatized by a knotted identity, a weak commutativity identity and any additional (possibly empty) set of equations over $\{\vee, \cdot, 1 \}$, then it has the FEP.
\end{theorem}

\begin{proof}
If $\mathcal{V}$ is such a variety, it suffices to prove the FEP for the subdirectly irrducible algebras in $\mathcal{V}$; so it suffices to prove it for unilinear residuated lattices.
Let $\mathbf{A}$ be a unilinear residuated lattice in $\mathcal{V}$ and $\mathbf{B}$ be a finite partial subalgebra of $\mathbf{A}$.

Let $\m W$ be the submonoid of $\m A$ generated by $B$, $W' = W \times B \times W$ and let $N \subseteq W \times W'$ be defined by: $x \mathrel{N} (y,b,z)$ if $yxz \leq b$.
Then $\mathbf{W_{A, B}} = (W, W', N, \cdot, 1)$ is a residuated frame in the sense of \cite{galatos2013residuated} and the Galois algebra $\mathbf{W_{A, B}}^+ = ({\gamma_N}[\mathcal{P}(W)], \cap, \cup_{\gamma_N}, \cdot_{\gamma_N}, \gamma(\{1\}), \ld, \rd)$ is a residuated lattice, where $X \cup_\gamma Y = \gamma(X \cup Y)$, $X \cdot_\gamma Y = \gamma(X \cdot Y)$, $X \ld Y = \{z \in W : zX \subseteq Y\}$ and $Y \rd X = \{z \in W : Xz \subseteq Y\}$.
Moreover, \cite{galatos2013residuated} shows that $\mathbf{W_{A, B}}^+$ satisfies all $\{\vee, \cdot, 1 \}$-equations that $\m A$ satisfies and that $\m B$ embeds in $\mathbf{W_{A, B}}^+$.
Also, \cite{cardona2017fep}  shows that such $\mathbf{W_{A, B}}^+$ is finite, due to the knotted rule and the weak commutativity.
So it suffices to show that it is in $\mathsf{SRL}$; we will show that $\mathbf{W_{A, B}}^+$ is actually unilinear.

Note that for all $(y, b, z) \in W'$, we have $a \in \{(y,b,z) \}^{\triangleleft}$ iff $a \mathrel{N} (y,b,z)$ iff $yaz \leq b$ iff $a \leq y \ld b \rd z$.
Therefore, $\{(y,b,z) \}^{\triangleleft} = \downarrow (y \ld b \rd z)$.
By basic properties of Galois connections, every element $X$ of ${\gamma_N}[\mathcal{P}(W)]$ is an intersection of sets of the form $\{(y,b,z)\}^{\triangleleft}$; actually $X = \bigcap \{\{w\}^{\triangleleft}: w \in X^{\triangleright}\}$.
Therefore, $X$ is an intersection of principal downsets of $\m A$.
Since $\m A$ is unilinear, $X$ is either equal to $A$ itself or a linear downset of $A$. 

Now, let $X, Y \in {\gamma_N}[\mathcal{P}(W)]$; hence each of them is either equal to $A$ or a linear subset of $A$.
If $X \nsubseteq Y$ and $Y \nsubseteq X$, then none of them equals $A$, hence they are both linear downsets.
Since $X \nsubseteq Y$, there is an $x \in X$ such that $x \not \in Y$.
Since, $Y \nsubseteq X$, not every element of $Y$ is below $x$, so there exists $y \in Y$ with $y \not \leq x$.
Since $x \not \in Y$ and $Y$ is a downset, we get $x \not \leq y$; therefore in this case $\m A$ is not linear.
By unilinearity of $\m A$, it has a top $\top$ and $\top = x \jn y \in X \cup_\gamma Y$, which is also a downset; hence $X \cup_\gamma Y = A$.
Also, if $z \in X \cap Y$, then $z \leq x, y$ and by the unilinearity of $\m A$, we get $z = \bot$; so $ X \cap Y = \{\bot\}$.
Consequently, ${\gamma_N}[\mathcal{P}(W)]$ is unilinear.
\end{proof}

Note that all knotted identities and all weak commutativity identities are equations over $\{\vee, \cdot, 1 \}$.
So, the theorem includes cases where multiple knotted and/or multiple weak commutativity equations are included in the axiomatization.

\begin{corollary}
    If a subvariety of $\mathsf{SRL}$ is axiomatized by a knotted identity, a weak commutativity identity and any (possibly empty) set of equations over $\{\vee, \cdot, 1 \}$, then its universal theory is decidable.
\end{corollary}

%%%%%%%%%%%%%%%%%%%%%%%%%%%%%%%%%%%%%%%%%%%%%%%%%%%%%%%%%%%%%%%%%%%%%%
\section{Constructing Compact URLs}\label{s: compact}

A unilinear residuated lattice $\m R$ is called \emph{compact} if it is \emph{$\top$-unital} (i.e., it satisfies: $x = \overline{\bot} \text{ or } x \overline{\top} = \overline{\top} = \overline{\top} x$) and $R \setminus \{\top, \bot\}$ is closed under multiplication.
In other words, non-linear compact URLs are obtained by a partially-ordered monoid $\m M$ that is a union of chains by adding bounds that absorb all elements of $M$.
We will provide some constuctions of compact URLs, but first we start by giving an axiomatization.

\begin{lemma}
The class of compact URLs is axiomatized by the sentences
 $\forall x (x = \overline{\bot} \text{ or } x \overline{\top} = \overline{\top} = \overline{\top} x)$ and
$\forall x, y, z \, ( x = \overline{\top} \text{ or } x (y \wedge z) = xy \wedge xz)$.
\end{lemma}

\begin{proof}
By the definition of compactness, it suffices to show that, for every $\top$-unital non-linear unilinear residuated lattice $\m R$, the second formula captures the fact that $R \setminus \{\top, \bot\}$ is closed under multiplication.
Note that if $a, b \not \in \{\top, \bot\}$, then $ab \top = a \top = \top$, so $ab \not = \bot$.  

Assume first that $\m R$ satisfies the second formula, but there exist $a_1, a_2 \in R \setminus \{\bot, \top\}$ such that $a_1 a_2 = \top$.
Since $\m R$ is not linear, there exists an element $a_3$ that is incomparable to $a_1$ or to $a_2$; without loss of generality, $a_3$ is incomparable to $a_2$, so $a_3 \in R \setminus \{\bot, \top\}$.
Hence
\begin{align*}
    \bot = a_1 \bot = a_1 (a_2 \wedge a_3) = a_1 a_2 \wedge a_1 a_3 = \top \wedge a_1 a_3 = a_1 a_3,
\end{align*}
a contradiction.
Thus $R \setminus \{\top, \bot\}$ is closed under multiplication.

Now assume $R \setminus \{\top, \bot\}$ is closed under multiplication and that $x, y, z \in R$ with $x \not = \top$.
If $x = \bot$, then the formula holds, so we assume that $x \not = \bot$. Also, if $y$ and $z$ are comparable, then $x (y \wedge z) = xy \wedge xz$ holds since multiplication preserves the order; so we assume that  $y$ and $z$ are incomparable.
In this case, $xy \vee xz =x (y \vee z) = x \top = \top$.
Since $R \setminus \{\top, \bot\}$ is closed under multiplication, $xy$ and $xz$ are incomparable, hence $x (y \wedge z) = x \cdot \bot = \bot = xy \wedge xz$.
\end{proof}

It follows that an alternative second formula is $\forall x,y,z \, ( x = \overline{\top} \text{ or } (y \wedge z) x = yx \wedge zx)$.

\begin{corollary}
The variety generated by the class of compact URL is axiomatized by
\begin{align*}
    1 = & \gamma_1(u \ld v) \jn \gamma_2(v \ld u) \jn \gamma_3(x \ld (u \mt v)) \jn \gamma_4((u \jn v) \ld (x(u \jn v) \mt (u \jn v)x))\\
    1 = & \gamma_5(u \ld v) \jn \gamma_6(v \ld u) \jn \gamma_7((u \jn v) \ld x) \jn \gamma_8((xu \mt xv) \ld x(u \mt v))
\end{align*}   
where $\gamma_1$,$\gamma_2$, $\gamma_3$, $\gamma_4$, $\gamma_5$,  $\gamma_6$,$\gamma_7$, $\gamma_8 \in \Gamma(Var)$.
\end{corollary}

\begin{lemma}\label{l: compact url}
If $\m R$ is a compact URL, then the comparability relation $\equiv$ on $\mathbf{M}$, where $M = R \setminus \{\bot, \top\}$, is a congruence relation and the quotient monoid $\mathbf{M} / {\equiv}$ is cancellative.
Also, $[1]_{\equiv}$ defines a totally-ordered submonoid of $\m M$.
\end{lemma}

\begin{proof}
That the comparability relation $\equiv$ is a congruence on $\mathbf{M}$ follows from the order-preservation of multiplication and the unilinear order.
For the cancellativity of $\mathbf{M} / {\equiv}$, note that if for $x, y, z \in M$ and $y \parallel z$, we have $\top = x (y \vee z) = xy \vee xz$, and since $\mathbf{M}$ is closed under multiplication, we get $xy \parallel xz$.
Finally, $[1]_{\equiv}$ is a totally-ordered submonoid of $\m M$ since $x \equiv 1$ and $y \equiv 1$ implies $xy \equiv 1 \cdot 1 = 1$.
\end{proof}

\subsection{From a finite cyclic monoid}\label{s: fin. cyclic monoid}

We show how to construct a compact URL starting from a finite cyclic monoid. 

Given a finite cyclic monoid $\mathbf{M}$ generated by an element $a$ of $M$, there is a smallest natural number $r$, called the \emph{index}, such that $a^r = a^{r+s}$ for some  positive integer $s$; the smallest such $s$ then is called the \emph{period}.
So $M = \{1, a, \dots, a^r, \dots, a^{r+s-1}\}$ and $|M|=r+s$.
Note that every natural number $n > r$ can be written as $n = r+ms+k$ for unique $m \in \mathbb{N}$ and $0 \leq k < s$; we define $[n]_r^s := r+k$ for $n \geq r+s$ and $[n]_r^s:= n$ for $0 \leq n < r+s$.
(We will write $[n]$, when $r,s$ are clear from the context.)
Then the multiplication on $\m M$ is given by $a^i \cdot a^j = a^{[i+j]_r^s}$.

In particular, $\{a^r, \dots, a^{r+s-1}\}$ is a subsemigroup of $\m M$ and it is a group in its own right with identity element $a^t$ such that $t \equiv 0 \, (\text{mod } s)$; so it is isomorphic to $\mathbb{Z}_s$.

We extend the multiplication of $\m M$ to the set $R = M \cup \{\bot, \top\}$ by $\bot x = x \bot = \bot$ for all $x \in R$, and $\top x = x \top = \top$ for all $x \not = \bot$.
Also we define an order on $R$ by $\bot \leq x \leq \top$ for all $x \in R$ and $a^i \leq a^j$ if and only if $j = i+ns$ for some $n \in \mathbb{N}$, where $0 \leq i, j \leq r+s-1$; see Figure~\ref{fin.cyclicmonid}(left).
It is easy to see that this yields a unilinear lattice order; we denote by $\m R_{\m M}$ the resulting lattice-ordered monoid.

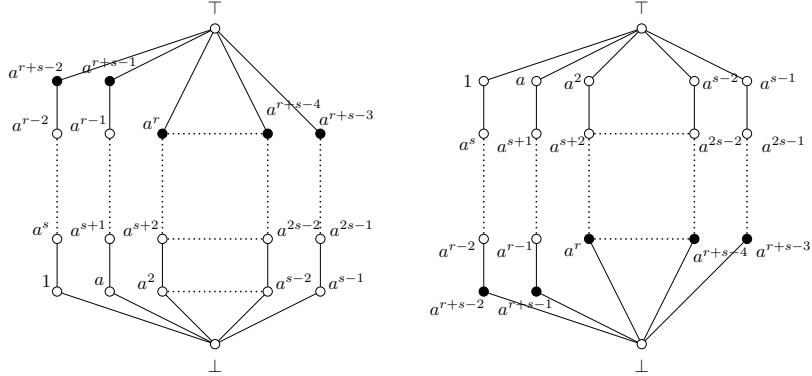
\begin{figure}[ht]
\centering
\scalebox{0.7}{
\begin{tikzpicture}
\draw (3, 3) -- (0, 2) -- (0, 1);
\draw [thick, dotted] (0, 1) -- (0, -1);
\draw (0, -1) -- (0, -2) -- (3, -3);
\draw (3, 3) -- (1, 2) -- (1, 1);
\draw [thick, dotted] (1, 1) -- (1, -1);
\draw (1, -1) -- (1, -2) -- (3, -3);
\draw (3, 3) -- (2, 1);
\draw [thick, dotted] (2, 1) -- (2, -1);
\draw (2, -1) -- (2, -2) -- (3, -3);
\draw [thick, dotted] (2, 1) -- (4, 1);
\draw [thick, dotted] (2, -1) -- (4, -1);
\draw [thick, dotted] (2, -2) -- (4, -2);
\draw (3, 3) -- (4, 1);
\draw [thick, dotted] (4, 1) -- (4, -1);
\draw (4, -1) -- (4, -2) -- (3, -3);
\draw (3, 3) -- (5, 1);
\draw [thick, dotted] (5, 1) -- (5, -1);
\draw (5, -1) -- (5, -2) -- (3, -3);

\filldraw [color = black, fill = white] (3, 3) circle (2.5pt)
    (3, 3.4) node {$\top$};
\filldraw [color = black, fill = white] (3, -3) circle (2.5pt)
    (3, -3.4) node {$\bot$};
\filldraw [color = black, fill = black] (0, 2) circle (2.5pt)
    (-0.4, 2.2) node {$a^{r+s-2}$};
\filldraw [color = black, fill = white] (0, 1) circle (2.5pt)
    (-0.5, 1.2) node {$a^{r-2}$};
\filldraw [color = black, fill = white] (0, -1) circle (2.5pt)
    (-0.3, -0.8) node {$a^s$};
\filldraw [color = black, fill = white] (0, -2) circle (2.5pt)
    (-0.2, -1.8) node {$1$};
    
\filldraw [color = black, fill = black] (1, 2) circle (2.5pt)
    (1, 2.3) node {$a^{r+s-1}$};
\filldraw [color = black, fill = white] (1, 1) circle (2.5pt)
    (0.6, 1.2) node {$a^{r-1}$};
\filldraw [color = black, fill = white] (1, -1) circle (2.5pt)
    (0.6, -0.8) node {$a^{s+1}$};
\filldraw [color = black, fill = white] (1, -2) circle (2.5pt)
    (0.8, -1.8) node {$a$};
    
\filldraw [color = black, fill = black] (2, 1) circle (2.5pt)
    (1.8, 1.2) node {$a^r$};
\filldraw [color = black, fill = white] (2, -1) circle (2.5pt)
    (1.6, -0.8) node {$a^{s+2}$};
\filldraw [color = black, fill = white] (2, -2) circle (2.5pt)
    (1.7, -1.8) node {$a^2$};
    
\filldraw [color = black, fill = black] (4, 1) circle (2.5pt)
    (4.5, 1.5) node {$a^{r+s-4}$};
\filldraw [color = black, fill = white] (4, -1) circle (2.5pt)
    (4.6, -0.8) node {$a^{2s-2}$};
\filldraw [color = black, fill = white] (4, -2) circle (2.5pt)
    (4.5, -1.8) node {$a^{s-2}$};
    
\filldraw [color = black, fill = black] (5, 1) circle (2.5pt)
    (5.5, 1.3) node {$a^{r+s-3}$};
\filldraw [color = black, fill = white] (5, -1) circle (2.5pt)
    (5.6, -0.8) node {$a^{2s-1}$};
\filldraw [color = black, fill = white] (5, -2) circle (2.5pt)
    (5.5, -1.8) node {$a^{s-1}$};
\end{tikzpicture}
\qquad
\begin{tikzpicture}
\draw (3, 3) -- (0, 2) -- (0, 1);
\draw [thick, dotted] (0, 1) -- (0, -1);
\draw (0, -1) -- (0, -2) -- (3, -3);
\draw (3, 3) -- (1, 2) -- (1, 1);
\draw [thick, dotted] (1, 1) -- (1, -1);
\draw (1, -1) -- (1, -2) -- (3, -3);
\draw (3, 3) -- (2, 2) -- (2, 1);
\draw [thick, dotted] (2, 1) -- (2, -1);
\draw (2, -1) -- (3, -3);
\draw [thick, dotted] (2, 1) -- (4, 1);
\draw [thick, dotted] (2, -1) -- (4, -1);
\draw (3, 3) -- (4, 2) -- (4, 1);
\draw [thick, dotted] (4, 1) -- (4, -1);
\draw (4, -1) -- (3, -3);
\draw (3, 3) -- (5, 2) -- (5, 1);
\draw [thick, dotted] (5, 1) -- (5, -1);
\draw (5, -1) -- (3, -3);

\filldraw [color = black, fill = white] (3, 3) circle (2.5pt)
    (3, 3.4) node {$\top$};
\filldraw [color = black, fill = white] (3, -3) circle (2.5pt)
    (3, -3.4) node {$\bot$};
\filldraw [color = black, fill = white] (0, 2) circle (2.5pt)
    (-0.3, 2) node {$1$};
\filldraw [color = black, fill = white] (0, 1) circle (2.5pt)
    (-0.3, 0.8) node {$a^s$};
\filldraw [color = black, fill = white] (0, -1) circle (2.5pt)
    (-0.5, -1.2) node {$a^{r-2}$};
\filldraw [color = black, fill = black] (0, -2) circle (2.5pt)
    (-0.5, -2.3) node {$a^{r+s-2}$};
    
\filldraw [color = black, fill = white] (1, 2) circle (2.5pt)
    (0.7, 2) node {$a$};
\filldraw [color = black, fill = white] (1, 1) circle (2.5pt)
    (0.6, 0.8) node {$a^{s+1}$};
\filldraw [color = black, fill = white] (1, -1) circle (2.5pt)
    (0.6, -1.2) node {$a^{r-1}$};
\filldraw [color = black, fill = black] (1, -2) circle (2.5pt)
    (0.8, -2.3) node {$a^{r+s-1}$};
    
\filldraw [color = black, fill = white] (2, 2) circle (2.5pt)
    (1.7, 2) node {$a^2$};
\filldraw [color = black, fill = white] (2, 1) circle (2.5pt)
    (1.6, 0.8) node {$a^{s+2}$};
\filldraw [color = black, fill = black] (2, -1) circle (2.5pt)
    (1.7, -1.2) node {$a^r$};
    
\filldraw [color = black, fill = white] (4, 2) circle (2.5pt)
    (4.5, 2) node {$a^{s-2}$};
\filldraw [color = black, fill = white] (4, 1) circle (2.5pt)
    (4.5, 0.8) node {$a^{2s-2}$};
\filldraw [color = black, fill = black] (4, -1) circle (2.5pt)
    (4.5, -1.3) node {$a^{r+s-4}$};
    
\filldraw [color = black, fill = white] (5, 2) circle (2.5pt)
    (5.6, 2) node {$a^{s-1}$};
\filldraw [color = black, fill = white] (5, 1) circle (2.5pt)
    (5.7, 0.8) node {$a^{2s-1}$};
\filldraw [color = black, fill = black] (5, -1) circle (2.5pt)
    (5.7, -1.2) node {$a^{r+s-3}$};
\end{tikzpicture}
}
\caption{The two URLs based on a finite cyclic monoid}
\label{fin.cyclicmonid}
\end{figure}

\begin{theorem}
If $\m M$ is a finite cyclic monoid, then $\mathbf{R}_{\m M}$ is the reduct of a residuated lattice.
\end{theorem}

\begin{proof}
    Since both $\top$ and $\bot$ are zero elements for $\mathbf{M}$ and $\bot \top = \top \bot = \bot$, the associativity of $\m M$ easily extends to the associativity of $\mathbf{R}_{\m M}$.
    Since $R$ is finite, by Corollary~\ref{c: RLjoin}, it suffices to show that multiplication distributes over binary joins; we will show distribution from the left: $x(y \vee z) = xy \vee xz$, for all $x, y, z \in R$.

    If any of $x, y, z$ is $\top$ or $\bot$, it easy to see that the equation holds, so we assume that $x, y, z \in M$: $x = a^i, y = a^j$ and $z = a^k$ for some $0 \leq i, j, k \leq r+s-1$.
    If $y = a^j$ and $z = a^k$ are incomparable, then $j \not \equiv k \, (\text{mod } s)$ by definition, so we have $i+j \not \equiv i+k \, (\text{mod } s)$ and hence $xy \parallel xz$.
    Thus we have $x(y \vee z) =x \top= \top= xy \vee xz$.
    If  $a^j = y \leq z = a_k$, we have $k = j+ns$ for some $0 \leq n \leq \lfloor(r+s-1-j)/s \rfloor$; we will show that $xy = a^{[i+j]} \leq a^{[i+j+ns]} = xz$.
    This is true since for $\ell = i+j$, we have $[\ell+ns] = [\ell] + ms$, where $m = n$ if $\ell+ns < r+s$ and $m = ([\ell+ns]-[\ell])/s$ if $\ell+ns \geq r+s$.
\end{proof}

The (commutative) residuated lattice based on $\mathbf{R}_{\m M}$ is compact so we have $\bot \ra x = \top =x \ra \top$, $\top \ra y = \bot$, $z \ra \bot = \bot $ for all $x \in R_M$, $y \not = \top$, $z \not = \bot$.
Also, the remaining implications can be easily calculated to be as follows: 
\begin{align*}
    a^i \rightarrow a^j =
    \begin{cases}
        \bot & \text{if } j < i \leq r  \text{ or } j< r \leq i \leq r+s-1\\
        a^{j-i} & \text{if } i \leq j < r\\
        a^{j-i+\lfloor \frac{r+s-1+i-j}{s} \rfloor s} & \text{if } i < r \leq j \leq r+s-1\\
        a^k & \text{if } r \leq i, j, k \leq r+s-1 \text{ and } a^i a^k = a^j.
    \end{cases}
\end{align*}
In particular, the subsemigroup $\{a^r, \dots, a^{r+s-1}\}$ is closed under implication, but $M$ is not.

It is easy to see that if we impose the dual order on the elements of $\m M$ instead, then we can obtain a different unilinear residuated lattice; see Figure~\ref{fin.cyclicmonid}(right).
Residuation in this second example works differently:
\begin{align*}
    a^i \rightarrow a^j =
    \begin{cases}
        a^{j-i} & \text{if } i \leq j \leq r+s-1\\
        a^{j-i+\lceil \frac{i-j}{s} \rceil s} & \text{if } j < i \leq r+s-1
    \end{cases}
\end{align*}
In this case, $M$ is closed under implication, but $\{a^r, \dots, a^{r+s-1}\}$ is not.

\begin{remark}
Actually, we can prove that given a finite cyclic monoid $M$, these are the only two ways where $M \cup \{\bot, \top\}$ is the monoid reduct of a compact unilinear residuated lattice.

Suppose $M \cup \{\bot, \top\}$ is the monoid reduct of a compact URL $\m R$.
Let $a^i$ and $a^j$ be distinct group elements in $M$.
If $a^i < a^j$, then $e = a^i a^k < a^j a^k$, where $e$ is the identity for the group elements in $M$ and $a^k$ is the inverse of $a^i$ in the group.
Then $e < a^j a^k < (a^j a^k)^2 < \cdots$, so $M$ contains an infinite ascending chain, contradicting the fact that $M$ is finite.
Thus the group elements in $M$ are pairwise incomparable.

We also observe that given $0 \leq i < j < r+s$,
\begin{equation}\tag{*}\label{eq: cyclic monoid *}
\begin{split}
    & a^i < a^j \text{ iff } \text{ for all } 0 \leq k \leq i, \, a^{i-k} < a^{j-k}\\
    & a^j < a^i \text{ iff } \text{ for all } 0 \leq k \leq i, \, a^{j-k} < a^{i-k}
\end{split}
\end{equation}
The backward direction is trivial, so we just show the forward direction.
Given $0 \leq i < j < r+s$ such that $a^i < a^j$ and $0 \leq k \leq i$, if $a^{i-k} \parallel a^{j-k}$, then $\top = a^k \top = a^k (a^{i-k} \jn a^{j-k}) = a^i \jn a^j$, so $a^i \parallel a^j$, a contradiction;
if $a^{i-k} > a^{j-k}$, then $a^i > a^j$ since multiplication is order-preserving and $a^i$ is distinct from $a^j$.

Finally, we know $1 \equiv e$, since otherwise we would have $\top = e(1 \vee e) = e \vee e^2 = e$, a contradiction.

Now let $t$ be the smallest natural number such that $a^t \equiv 1$.
If $t = 0$, then by (\ref{eq: cyclic monoid *}), $a^i \parallel a^j$ for all $0 \leq i < j < r+s$; otherwise $1 \equiv a^{j-i}$ where $j-i>0$, a contradiction.
Especially we have $e = 1$ in this case, so $M$ is a group and $\m R$ is based on $\m M_X$.
Now we assume $t > 0$.
If $1 < a^t$, then we have $a^r \leq a^{r+t}$ and both of them are group elements in $M$.
Since all group elements are pairwise incomparable, we know $a^r = a^{r+t}$, so $t = s$.
Since $s = t$ is the smallest integer such that $1 < a^s$, we know $1 \parallel a^k$ for all $1 < k < s$, thus by (\ref{eq: cyclic monoid *}) $a^k \parallel a^l$ for all $0 \leq k \neq l \leq s-1$.
Since $1 < a^s$, we have $1 < a^s < a^{2s} < \cdots < a^{ms}$, where $ms < r+s \leq (m+1)s$.
Hence $a^i < a^j$ iff $1 < a^{j-i}$ iff $j = i+ns$ for some $n \in \bb{Z}^+$, so $a^i \leq a^j$ iff $j = i+ns$ for some $n \in \bb{N}$ and $\m R$ is of the form as the left in Figure~\ref{fin.cyclicmonid}(left).
Similarly we can prove $\m R$ is of the form as the right in Figure~\ref{fin.cyclicmonid}(right) if $a^t < 1$.
\end{remark}

\subsection{From a semidirect product of a residuated chain and a cancellative monoid; monoid extensions with 2-cocycles}

We first provide a general construction of compact residuated lattices and then show that under certain assumptions a compact residuated lattice is exactly of this form.  

Let $\mathbf{A}$ be a residuated chain, $\mathbf{K}$ a cancellative monoid and $\varphi: \mathbf{K} \rightarrow \m{ResEnd}(\mathbf{A})$ a monoid homomorphism, where $\m{ResEnd}(\mathbf{A})$ is the monoid of residuated maps on the chain $(A, \leq)$ which are also endomorphisms of the monoid $(A, \cdot, 1)$.
If $\varphi, \psi \in \m {ResEnd} (\m A)$ with residuals $\varphi^*$ and $\psi^*$, then $(\psi \circ \varphi)(a) \leq b$ iff $\varphi(a) \leq \psi^*(b)$ iff $a \leq (\varphi^* \circ \psi^*)(b)$ for all $a, b \in A$; so $\psi \circ \varphi$ is also residuated.
Thus, $\m{ResEnd}(\mathbf{A})$ is a submonoid of $\m {End}(\m A)$.
Consequently, the semidirect product $\mathbf{A} \rtimes_{\varphi} \mathbf{K}$ of the monoid reduct of $\mathbf{A}$ and $\mathbf{K}$ with respect to $\varphi$ is also a monoid with multiplication given by
\begin{align*}
    (a_1, k_1) \cdot (a_2, k_2) = (a_1 \varphi_{k_1}(a_2), k_1 k_2),
\end{align*}
for all $(a_1, k_1), (a_2, k_2) \in A \times K$, and identity $(1_{\mathbf{A}}, 1_{\mathbf{K}})$.
We define an order on $\mathbf{A} \rtimes_{\varphi} \mathbf{K}$ by: for all $(a_1, k_1), (a_2, k_2) \in A \times K$,
\begin{align*}
    (a_1, k_1) \leq (a_2, k_2) \text{ if and only if } k_1 = k_2 \text{ and } a_1 \leq a_2.
\end{align*}
Also, we extend the multiplication and order of $\mathbf{A} \rtimes_{\varphi} \mathbf{K}$ to $R = (A \times K) \cup \{\top, \bot\}$ by: $\bot \leq x \leq \top$, $\bot x = x \bot = \bot$ and $\top y = y \top = \top$ for all $x \in R$, $y \not = \bot$.
It is clear that this defines a lattice order; see Figure~\ref{f: semidirect}.
We denote by $\mathbf{A} \rtimes^b_{\varphi} \mathbf{K}$ the resulting bounded lattice-ordered monoid.

\begin{figure}
\centering
\scalebox{0.9}{
    \begin{tikzpicture}
        \draw (3, 3) -- (0, 2);
        \draw[dotted] (0, 2) -- (0, -2);
        \draw (0, -2) -- (3, -3);
        \draw (3, 3) -- (1.5, 2);
        \draw[dotted] (1.5, 2) -- (1.5, -2);
        \draw (1.5, -2) -- (3, -3);
        \draw (3, 3) -- (5, 2);
        \draw[dotted] (5, 2) -- (5, -2);
        \draw (5, -2) -- (3, -3);
        \draw[dotted] (2, 0) -- (4.5, 0);

        \filldraw [color = black, fill = white] (3, 3) circle(2.5pt)
            (3, 3.4) node {$\top$};
        \filldraw [color = black, fill = white] (3, -3) circle(2.5pt)
            (3, -3.4) node {$\bot$};
        \filldraw [color = black, fill = white] (0, 1) circle(2.5pt)
            (-0.8, 1) node {$(a_1, k_1)$};
        \filldraw [color = black, fill = white] (0, 0) circle(2.5pt)
            (-0.3, 0) node {$1$};
        \filldraw [color = black, fill = white] (0, -1) circle(2.5pt)
            (-0.8, -1) node {$(a_2, k_1)$};

        \filldraw [color = black, fill = white] (1.5, 1) circle(2.5pt)
            (2.3, 1) node {$(a_1, k_2)$};
        \filldraw [color = black, fill = white] (1.5, 0) circle(2.5pt);
        \filldraw [color = black, fill = white] (1.5, -1) circle(2.5pt)
            (2.3, -1) node {$(a_2, k_2)$};

        \filldraw [color = black, fill = white] (5, 1) circle(2.5pt);
        \filldraw [color = black, fill = white] (5, 0) circle(2.5pt);
        \filldraw [color = black, fill = white] (5, -1) circle(2.5pt);
    \end{tikzpicture}
}
\caption{A URL based on a semidirect product}
\label{f: semidirect}
\end{figure}
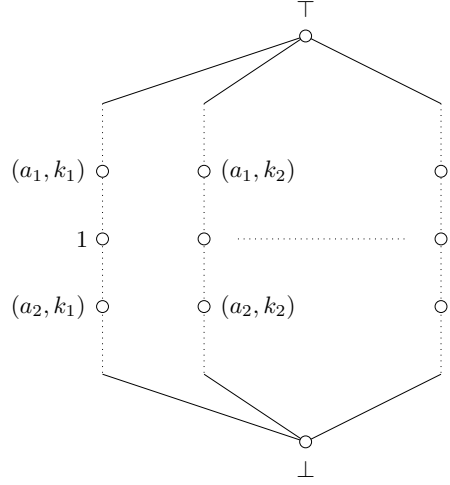

\begin{theorem}\label{construction_on_semidirect_product}
  If $\mathbf{A}$ is a residuated chain, $\mathbf{K}$ is a cancellative monoid and $\varphi: \mathbf{K} \rightarrow \m{ResEnd}(\mathbf{A})$ is a monoid homomorphism, then $\mathbf{A} \rtimes^b_{\varphi} \mathbf{K}$ is a residuated lattice.
\end{theorem}

The proof of the above theorem follows from a more general construction.
Given a monoid $\mathbf{K}$, a totally-ordered monoid $\mathbf{A}$ and a map $\varphi: \mathbf{K} \rightarrow \m {ResEnd}(\mathbf{A})$, then a function $f: K \times K \rightarrow A$ is called a \emph{$2$-cocycle}
with respect to $\m K, \m A, \varphi$, if it satisfies the following conditions:
\begin{enumerate}
    \item $f(k_1, k_2)$ is invertible, for all $k_1, k_2 \in K$.
    
    \item $f(k, 1) = f(1, k) = 1$, for all $k \in K$.

    \item $\varphi_{1_{\mathbf{K}}} = \text{id}_{\mathbf{A}}$ and $\varphi_{k_1 k_2}(a) = f(k_1, k_2) \cdot \varphi_{k_1}\varphi_{k_2}(a) \cdot f(k_1, k_2)^{-1}$ for all $k_1, k_2 \in K$ and $a \in A$.
    
    \item $f(k_1, k_2 k_3) \varphi_{k_1}(f(k_2, k_3)) = f(k_1 k_2, k_3) f(k_1, k_2)$, for $k_1, k_2, k_3 \in K$.
\end{enumerate}

Now, given a cancellative monoid $\mathbf{K}$, a residuated chain $\mathbf{A}$, a map $\varphi: K \rightarrow \m {ResEnd}(\mathbf{A})$ and a $2$-cocycle $f: K \times K \rightarrow A$, we define multiplication on $A \times K$ by 
\begin{align*}
    (a_1, k_1) \cdot (a_2, k_2) = (a_1 \varphi_{k_1}(a_2) f(k_1, k_2)^{-1}, k_1 k_2)
\end{align*}
Also, we extend the multiplication to $R = A \times K \cup \{\bot, \top\}$ by making $\bot$ absorbing for $R$ and $\top$ absorbing for $R \setminus\{\bot\}$, 
and we define a lattice ordering $\leq$ by: for all $a, a_1, a_2 \in A$ and $k, k_1, k_2 \in K$, $\bot = \bot < (a, k) < \top = \top$ and
\[
    (a_1, k_1) \leq (a_2, k_2) \text{ iff } a_1 \leq_{\mathbf{A}} a_2 \text{ and } k_1 = k_2.
\] 
We denote the resulting algebra by $\mathbf{R}_{\varphi, f}$.

\begin{theorem}
    If $\mathbf{K}$ is a cancellative monoid, $\mathbf{A}$ is a residuated chain, $\varphi: \mathbf{K} \rightarrow \m {ResEnd}(\mathbf{A})$ is a map, and $f: K \times K \rightarrow A$ is a $2$-cocycle with respect to $\m K$, $\m A$ and $\varphi$, then $\mathbf{R}_{\varphi, f}$ is the reduct of a residuated lattice.
\end{theorem}

\begin{proof}
    In the following we use $\m R$ for $\mathbf{R}_{\varphi, f}$ and $M$ for $A \times K$.
    Clearly, $M$ is closed under multiplication and $(1, 1)$ is the identity.
    Also, 
    \begin{align*}
        & (a_1, k_1) (a_2, k_2) \cdot (a_3, k_3)\\ 
        = & (a_1 \varphi_{k_1}(a_2) f(k_1, k_2)^{-1}, k_1 k_2) \cdot (a_3, k_3)\\
        = & (a_1 \varphi_{k_1}(a_2) f(k_1, k_2)^{-1} \varphi_{k_1 k_2}(a_3) f(k_1 k_2, k_3)^{-1}, k_1 k_2 \cdot k_3)\\
        = & (a_1 \varphi_{k_1}(a_2) f(k_1, k_2)^{-1} \cdot f(k_1, k_2) \varphi_{k_1}\varphi_{k_2}(a_3) f(k_1, k_2)^{-1} \cdot f(k_1 k_2, k_3)^{-1},\\
        &  k_1 k_2 \cdot k_3)\\
        = & (a_1 \varphi_{k_1}(a_2) \varphi_{k_1}\varphi_{k_2}(a_3) f(k_1, k_2)^{-1} f(k_1 k_2, k_3)^{-1}, k_1 k_2 \cdot k_3)\\
        = & (a_1 \varphi_{k_1}(a_2) \varphi_{k_1}\varphi_{k_2}(a_3) \varphi_{k_1}(f(k_2, k_3)^{-1}) f(k_1, k_2 k_3)^{-1}, k_1 k_2 \cdot k_3)\\
        = & (a_1 \varphi_{k_1}(a_2 \varphi_{k_2}(a_3) f(k_2, k_3)^{-1}) f(k_1, k_2 k_3)^{-1}, k_1 \cdot k_2 k_3)\\
        = & (a_1, k_1) \cdot (a_2 \varphi_{k_2}(a_3) f(k_2, k_3)^{-1}, k_2 k_3)\\
        = & (a_1, k_1) \cdot (a_2, k_2) (a_3, k_3)
    \end{align*}
    where we used the identities
    \begin{gather*}
        \varphi_{k_1 k_2}(a) = f(k_1, k_2) \cdot \varphi_{k_1}\varphi_{k_2}(a) \cdot f(k_1, k_2)^{-1}\\
        f(k_1, k_2 k_3) \varphi_{k_1}(f(k_2, k_3)) = f(k_1 k_2, k_3) f(k_1, k_2)
    \end{gather*}
    and the assumption that $\varphi_k$ is an endomorphism.
    Therefore $\mathbf{M} = (M, \cdot, (1,1))$ is a monoid.
    Since both $\top$ and $\bot$ are absorbing elements for $\m M$ and $\top \bot = \bot \top = \bot$, associativity holds on $\m R$.
    
    We now prove that multiplication is order-preserving:
    $y \leq z \implies (xy \leq xz \text{ and } yx \leq zx)$ for all $x, y, z \in R$.
    If $y = z$ or $x, y, z$ is $\bot$ or $\top$, then it it easy to see that the implication holds; so we assume that $\bot < x < \top$ and $\bot < y < z < \top$.
    Also, we assume that $x = (a_1, k_1)$, $y = (a_2, k_2)$ and $z = (a_3, k_2)$ with $a_2 < a_3$.
    Using the order preservation of $\varphi_{k_1}$ (it is a residuated map) and of multiplication in $\mathbf{A}$, we get
    \begin{align*}
        (a_1, k_1) (a_2, k_2) & = (a_1 \varphi_{k_1}(a_2) f(k_1, k_2)^{-1}, k_1 k_2)\\
        & \leq (a_1 \varphi_{k_1}(a_3) f(k_1, k_2)^{-1}, k_1 k_2)\\
        & = (a_1, k_1) (a_3, k_2)\\
        (a_2, k_2) (a_1, k_1) & = (a_2 \varphi_{k_2}(a_1) f(k_2, k_1)^{-1}, k_2 k_1)\\
        & \leq (a_3 \varphi_{k_2}(a_1) f(k_2, k_1)^{-1}, k_2 k_1)\\
        & = (a_3, k_2) (a_1, k_1)
    \end{align*}
    Next we show that the sets $x \ldd z$ and $z \rdd x$ have maximum elements for all $x, z \in R$.
    By Remark~\ref{bottopdivison}, we know $\bot \ldd z = z \rdd \bot = x \ldd \top = \top \rdd x = R$ for all $x, z \in R$, so the maximum element of all of these sets is $\top$.
    Also, by construction, $x \ldd \bot = \bot \rdd x = \top \ldd z = z \rdd \top = \{\bot\}$ for all $x \in R \setminus \{\bot\}$ and $z \in R \setminus \{\top\}$, so the maximum for all these sets is $\bot$.
    We now assume that $\bot < x, z < \top$ and that $x = (a, k)$ and $z = (a', k')$ for some $(a, k), (a', k') \in A \times K$.

    For all $(a_1, k_1), (a_2, k_2) \in x \ldd z$, we have $(a \varphi_{k}(a_1) f(k, k_1)^{-1}, k k_1) = (a, k) (a_1, k_1) \leq (a', k')$ and $(a \varphi_{k}(a_2) f(k, k_2)^{-1}, k k_2) = (a, k) (a_2, k_2) \leq (a', k')$, so $kk_1 = k' = kk_2$ and $k_1 = k_2$, by the cancellativity of $\m K$.
    Since, $\m A$ is a chain, we get that $(a_1, k_1)$ and $(a_2, k_2)$ are comparable; hence $x \ldd z$ is a chain.
    
    For all $(a'', k'')$, we have that $(a'', k'') \in x \ldd z$ iff $(a, k) (a'', k'') \leq (a', k')$ iff 
    $(a \varphi_k(a'') f(k, k'')^{-1}, k k'') \leq (a', k')$ iff
    ($a \varphi_k(a'') f(k, k'')^{-1} \leq a'$ and $k' = k k''$).
    Since multiplication is residuated, $\varphi_k$ is residuated, say with residual $\varphi_k^*$, and $f(k, k'')$ is invertible, we have: $a \varphi_k(a'') f(k, k'')^{-1} \leq a'$ iff
    $\varphi_k(a'') \leq a \backslash_{\mathbf{A}} a' f(k, k'')$ iff $a'' \leq \varphi_k^*(a \backslash_{\mathbf{A}} a' f(k, k''))$.
    Therefore, we have $(a'', k'') \in x \ldd z$ iff $(a'', k'') \leq (\varphi_k^* (a \backslash_{\mathbf{A}} a' f(k, k'')), k'')$.
    Consequently, $\max(x \ldd z)$ exists and it is one of the elements $\bot, (\varphi_k^*(a \backslash_{\mathbf{A}} a' f(k, k'')), k''), \top$.
    Likewise, $\max(z \rdd x)$ is one of the elements $\bot, (a' f(k'', k) /_{\mathbf{A}} \varphi_{k''}(a), k''), \top$.
   By Corollary~\ref{c: RLmax}, $\mathbf{R}_{\varphi, f}$ is the reduct of a compact residuated lattice.
\end{proof}

So $\mathbf{R}_{\varphi, f}$ is the reduct of a compact residuated lattice, which we will also denote by $\mathbf{R}_{\varphi, f}$ and whose divisions are given by
\begin{align*}
    x \backslash y & =
    \begin{cases}
        \bot & \text{if } x = (a_1, k_1), y = (a_2, k_2) \text{ and } k_2 \notin k_1 K\\
        (\varphi_{k_1}^*(a_1 \backslash_{\mathbf{A}} a_2 f(k_1, k)), k) & \text{if } x = (a_1, k_1), y = (a_2, k_1 k)
    \end{cases}\\
    y / x & =
    \begin{cases}
        \bot & \text{if } x = (a_1, k_1), y = (a_2, k_2) \text{ and } k_2 \notin K k_1\\
        (a_2 f(k, k_1) /_{\mathbf{A}} \varphi_k(a_1), k) & \text{if } x = (a_1, k_1), y = (a_2, kk_1)
    \end{cases}
\end{align*}
and the standard divisions involving $\bot$ and $\top$ are given by Remark~\ref{bottopdivison}.

Theorem~\ref{construction_on_semidirect_product} follows as the special case where the $2$-cocycle is trivial, thus implying that $\varphi$ is a monoid homomorphism.
 
\begin{corollary}
If $\mathbf{K}$ is a cancellative monoid, $\mathbf{A}$ is a residuated chain, $\varphi: \mathbf{K} \rightarrow \m {ResEnd}(\mathbf{A})$ is a map, and $f: K \times K \rightarrow A$ is the trivial $2$-cocycle with respect to $\m K$, $\m A$ and $\varphi$, then $\varphi$ is a homomorphism and $\m R_{\varphi, f} = \m A \rtimes^b_{\varphi} \m K$.
\end{corollary}

In particular, when $\varphi$ is trivial we get $\m A \times^b \m K$, where $\m A$ is a residuated chain and $\m K$ is a cancellative monoid.

Note that the examples of section~\ref{s: fin. cyclic monoid} are not embeddable into a residuated lattice of the form $\m A \times^b \m K$.
For example, consider the URL $\m R$ where $R = \{\bot, 1, a, a^2, \top\}$ with $a^3 = a$ and $1 < a^2$.
If $\m R$ were embeddable then we would have $1 \mapsto (1,1)$, $a \mapsto (a_1, k)$, $a^2 \mapsto (a_1^2, k^2)$ and $a^3 \mapsto (a_1^3, k^3)$.
So, $(1, 1) < (a_1^2, k^2)$ implies $1 < a_1^2$ and $k^2 = 1$; thus $1 < a_1$ and $k = 1$.
But then $a_1 \leq a_1^2 \leq a_1^3 = a_1$, so $a_1^2 = a_1$, hence $(a_1^2, k^2) = (a_1, k)$, a contradiction.

Even though not all compact URLs are of the form $\m R_{\varphi,f}$, we show that this holds when the comparability relation on $R \setminus \{\bot, \top\}$ is an \emph{admissible} congruence and the chain of $1$ is \emph{cancellative with respect to} the factor monoid. 

We say that the congruence $\equiv$ on $\m M$ is \emph{admissible} if $x[1]_{\equiv} = [x]_{\equiv} =[1]_{\equiv}x$, for all $x \in M$. 
Also, we say $\mathbf{H}$ is \emph{$\mathbf{K}$-cancellative} if there exists a selection of representatives $^-: K \rightarrow M$ (i.e., for all $x \in M$, if $\overline{k} \equiv x$ then $x \in k$) satisfying $\overline{1_{\mathbf{K}}} = 1_{\m M}$ and the left and right multiplications by $\overline{k}$ are injective on $H$.
The terminology $\m K$-cancellative and $2$-cocycle come from \cite{pavel2002monoid}.

\begin{prop}
    If $\mathbf{R}$ is a compact unilinear residuated lattice, the comparability relation $\equiv$ is an admissible congruence of $\mathbf{M}$, where $M = R \setminus \{\bot, \top\}$, and  $\mathbf{H}$ is $\mathbf{K}$-cancellative, where $H = [1]_{\equiv}$ and $\mathbf{K} = \mathbf{M}/{\equiv}$, then $\mathbf{R} \cong \m R_{\varphi, f}$ for some map $\varphi:K \rightarrow \m{ResAut}(\m H)$ and $2$-cocycle $f: K \times K \ra H$ with respect to $\m H$, $\m K$ and $\varphi$.
\end{prop}

\begin{proof}
    Since $\mathbf{H}$ is $\mathbf{K}$-cancellative, there exists a selection of representatives $^-: K \rightarrow M$.
    We denote by $L_x$ and $R_x$ the left and right multiplication by $x \in M$, respectively.
    We know that for all $k \in K$, the maps $R_{\overline{k}}, L_{\overline{k}}: H \ra k$ are injective and since $\equiv$ is an admissible congruence on $\mathbf{M}$ and $H = [1]_{\equiv}$, they are also surjective. 
    So, for any $k \in K$, the map $\varphi_k: \mathbf{H} \rightarrow \mathbf{H}$ given by $\varphi_k(h) = R_{\overline{k}}^{-1} L_{\overline{k}}(h)$ is a well-defined bijection on $H$; hence $\overline{k} h = \varphi_k(h) \overline{k}$.
    
    Note that
    \begin{align*}
        \varphi_k(h_1 h_2) \overline{k} & = \overline{k} \cdot h_1 h_2\\
        & = \overline{k} h_1 \cdot h_2\\
        & = \varphi_k(h_1) \overline{k} \cdot h_2\\
        & = \varphi_k(h_1) \cdot \overline{k} h_2\\
        & = \varphi_k(h_1) \cdot \varphi_k(h_2) \overline{k}\\
        & = \varphi_k(h_1) \varphi_k(h_2) \cdot \overline{k}.
    \end{align*}
    Since $\mathbf{H}$ is $\mathbf{K}$-cancellative, we have $\varphi_k(h_1 h_2) = \varphi_k(h_1) \varphi_k (h_2)$.
    Now suppose $h_1 \leq h_2$ for some $h_1, h_2 \in H$.
    Since $\mathbf{R}$ is residuated, we get
    \begin{align*}
        \varphi_k(h_1) \leq \varphi_k(h_2) & \text{ iff } R_{\overline{k}}^{-1} L_{\overline{k}}(h_1) \leq R_{\overline{k}}^{-1} L_{\overline{k}}(h_2)\\
        & \text{ iff } R_{\overline{k}}(R_{\overline{k}}^{-1} L_{\overline{k}} (h_1)) \leq L_{\overline{k}}(h_2)\\
        & \text{ iff } L_{\overline{k}}(h_1) \leq L_{\overline{k}}(h_2).
    \end{align*}
    It follows from the order-preservation of $L_{\overline{k}}$ that $\varphi_k$ is order-preserving.
    So $\varphi_k$ is an automorphism of the totally-ordered monoid $\mathbf{H}$, and $\overline{1_{\mathbf{K}}} = 1_{\mathbf{H}}$ yields $\varphi_{1} = \text{id}_{\mathbf{H}}$.
    
    Since $\equiv$ is admissible on $\mathbf{M}$ and $\mathbf{K} = \mathbf{M}/\equiv$, we have
    \begin{align*}
        \mathbf{H} \overline{k_1 k_2} = k_1 k_2 = \mathbf{H} \overline{k_1} \mathbf{H} \overline{k_2} = \mathbf{H} \overline{k_1} \, \overline{k_2}.
    \end{align*}
    Therefore there exist $f(k_1, k_2)$ and $g(k_1, k_2)$ in $H$ such that
    $$ \overline{k_1 k_2} = f(k_1, k_2) \overline{k_1} \, \overline{k_2}, \qquad  \overline{k_1} \, \overline{k_2} = g(k_1, k_2) \overline{k_1 k_2}$$
    for all $k_1, k_2 \in K$.
    Since $\mathbf{H}$ is $\mathbf{K}$-cancellative, it follows that $f$ and $g$ are well-defined functions from $K \times K$ to $H$.
    Moreover, since $f(k_1, k_2) g(k_1, k_2) = g(k_1, k_2) f(k_1 k_2) = 1$ for all $k_1, k_2 \in K$, we get that $f(k_1, k_2)$ and $g(k_1, k_2)$ are invertible.
    By definition, we have
    $$ \overline{k_2} = f(1_{\mathbf{K}}, k_2) \overline{1_{\mathbf{K}}} \, \overline{k_2}, \qquad  \overline{k_1} = f(k_1, 1_{\mathbf{K}}) \overline{k_1} \, \overline{1_{\mathbf{K}}}.$$
    Again by the $\mathbf{K}$-cancellativity of $\mathbf{H}$, we get $f(1_{\mathbf{K}}, k) = f(k, 1_{\mathbf{K}}) = 1_{\mathbf{H}}$ for all $k \in K$.
    Also, by the definition of $f$, we know
    $$L_{\overline{k_1 k_2}} = L_{f(k_1, k_2)} L_{\overline{k_1}} L_{\overline{k_2}}, \qquad
        R_{\overline{k_1 k_2}} = R_{\overline{k_2}} R_{\overline{k_1}} R_{f(k_1, k_2)}.$$
    Thus by the $\mathbf{K}$-cancellativity of $\mathbf{H}$ we have
    \begin{align*}
        \varphi_{k_1 k_2} & = R^{-1}_{\overline{k_1 k_2}} L_{\overline{k_1 k_2}}\\
        & = R^{-1}_{f(k_1, k_2)} R^{-1}_{\overline{k_1}} R^{-1}_{\overline{k_2}} L_{f(k_1, k_2)} L_{\overline{k_1}} L_{\overline{k_2}}\\
        & = R^{-1}_{f(k_1, k_2)} L_{f(k_1, k_2)} R^{-1}_{\overline{k_1}} L_{\overline{k_1}} R^{-1}_{\overline{k_2}} L_{\overline{k_2}}\\
        & = R^{-1}_{f(k_1, k_2)} L_{f(k_1, k_2)} \varphi_{k_1} \varphi_{k_2}
    \end{align*}
    for all $k_1, k_2 \in K$.
    So we get
    \begin{align*}
        \varphi_{k_1 k_2}(h) = f(k_1, k_2) \cdot_{\mathbf{H}} \varphi_{k_1}\varphi_{k_2}(h) \cdot_{\mathbf{H}} f(k_1, k_2)^{-1}
    \end{align*}
    for all $h \in H$.
    
    Finally, we observe that
    \begin{align*}
       & \overline{k_1 \cdot k_2 k_3} = \overline{k_1 k_2 \cdot k_3}\\
        \text{ iff } & f(k_1, k_2 k_3) \overline{k_1} \, \overline{k_2 k_3} = f(k_1 k_2, k_3) \overline{k_1 k_2} \, \overline{k_3}\\
         \text{ iff } & f(k_1, k_2 k_3) \overline{k_1} f(k_2, k_3) \overline{k_2} \, \overline{k_3} = f(k_1 k_2, k_3) f(k_1, k_2) \overline{k_1} \, \overline{k_2} \cdot \overline{k_3}\\
        \text{ iff } & f(k_1, k_2 k_3) \varphi_{k_1}(f(k_2, k_3)) \overline{k_1} \cdot \overline{k_2} \, \overline{k_3} = f(k_1 k_2, k_3) f(k_1, k_2) \overline{k_1} \, \overline{k_2} \cdot \overline{k_3}.
    \end{align*}
    So by the associativity of $\mathbf{K}$ and the $\mathbf{K}$-cancellativity of $\mathbf{H}$, we get
    \begin{align*}
        f(k_1, k_2 k_3) \varphi_{k_1}(f(k_2, k_3)) = f(k_1 k_2, k_3) f(k_1, k_2)
    \end{align*}
    for all $k_1, k_2, k_3 \in K$.
    Therefore $f$ is a $2$-cocycle with respect to $\m H$, $\m K$ and $\varphi$.

    Finally, we define the map $\psi: \mathbf{R} \ra \m R_{\varphi, f}$, given by $\psi(\bot) = \bot$, $\psi(\top) = \top$ and $\psi(x)=(h_x, k_x)$, where $k_x = [x]_\equiv$ is the chain to which $x$ belongs and $h_x = R_{\overline{k_x}}^{-1}(x)$.
    Since $\equiv$ is admissible, $\m H$ is $\m K$-cancellative and $H$ is totally-ordred, $L_{\overline{k_x}}$ and $R_{\overline{k_x}}$ are order isomorphisms between the sets $H$ and $k_x$, so $\psi$ is well-defined.
    We will show that $\psi$ is a residuated-lattice isomorphism.

    Suppose $\psi(x) = \psi(y)$ for some $x, y \in M$.
    Then $k_x = k_y$ and $h_x = h_y$, i.e., $x \equiv y$ and $R_{\overline{k_x}}^{-1}(x) = R_{\overline{k_y}}^{-1}(y)$.
    Since $R_{\overline{k_x}} = R_{\overline{k_y}}$ is a bijection between $H$ and $k_x$, we have $x = y$.
    For $(h, k) \in H \times K$, let $x = h \overline{k}$.
    Since $R_{\overline{k}}$ is a bijection, we know $h = R_{\overline{k}}^{-1}(x)$, so $\psi(x) = (h, k)$.
    Since $\psi(\bot) = \bot$ and $\psi(\top) = \top$ are uniquelly defined, $\psi$ is a bijection between $R$ and $R_{\varphi, f}$.
    
    Since $R_{\overline{k_x}}$ is an order isomorphism between $\m H$ and the chain $k_x$, $x \leq_{\m R} y$ iff $k_x = k_y$ and $R_{\overline{k_x}}^{-1}(x) \leq_{\m H} R_{\overline{k_y}}^{-1}(y)$, hence $x \leq_{\m R} y$ iff $\psi(x) \leq_{\m R_{\varphi, f}} \psi(y)$ for all $x, y \in M$.
    Since $\psi(\bot) = \bot$ and $\psi(\top) = \top$, $\psi$ is a lattice isomorphism between $\m R$ and $\m R_{\varphi, f}$.
    
    Since $k_{xy} = k_x \cdot_{\m K} k_y$, we have $\overline{k_{xy}} = \overline{k_x k_y} = f(k_x, k_y) \overline{k_x} \, \overline{k_y}$, so for all $x, y \in M$
    \[
    \psi(xy) = (R_{\overline{k_{xy}}}^{-1}(xy), k_{xy}) = (R_{\overline{k_x}}^{-1} R_{\overline{k_y}}^{-1}(xy) f^{-1}(k_x, k_y), k_x k_y).
    \]
    On the other hand,
    \begin{align*}
        \psi(x) \psi(y) = & (R_{\overline{k_x}}^{-1}(x), k_x) (R_{\overline{k_y}}^{-1}(y), k_y)\\
        = & (R_{\overline{k_x}}^{-1}(x) \varphi_{k_x}(R_{\overline{k_y}}^{-1}(y)) f^{-1}(k_x, k_y), k_x k_y)\\
        = & (\varphi_{k_x}(L_{\overline{k_x}}^{-1}(x)) \varphi_{k_x}(R_{\overline{k_y}}^{-1}(y)) f^{-1}(k_x, k_y), k_x k_y)\\
        = & (\varphi_{k_x}(L_{\overline{k_x}}^{-1}(x) R_{\overline{k_y}}^{-1}(y)) f^{-1}(k_x, k_y), k_x k_y)
    \end{align*}
    Since
    \[xy = L_{\overline{k_x}} L_{\overline{k_x}}^{-1}(x) \cdot R_{\overline{k_y}} R_{\overline{k_y}}^{-1}(y) = R_{\overline{k_y}} L_{\overline{k_x}}(L_{\overline{k_x}}^{-1}(x) R_{\overline{k_y}}^{-1}(y)),
    \]
    we have that
    \[
    L_{\overline{k_x}}^{-1}(x) R_{\overline{k_y}}^{-1}(y) = L_{\overline{k_x}}^{-1} R_{\overline{k_y}}^{-1}(xy).
    \]
    So
    \[R_{\overline{k_x}}^{-1} R_{\overline{k_y}}^{-1}(xy) = R_{\overline{k_x}}^{-1} (L_{\overline{k_x}} L_{\overline{k_x}}^{-1}) R_{\overline{k_y}}^{-1}(xy) = \varphi_{k_x}(L_{\overline{k_x}}^{-1} R_{\overline{k_y}}^{-1}(xy)) = \varphi_{k_x}(L_{\overline{k_x}}^{-1}(x) R_{\overline{k_y}}^{-1}(y)),
    \]
    hence
    \[\psi(xy) = \psi(x) \psi(y).\]
    Since $\m R$ is compact, we know $\psi(xy) = \psi(x) \psi(y)$ for all $x, y \in R$.
    So $\psi$ is a lattice-ordered monoid isomorphism.
    Since both of $\m R$ and $\m R_{\varphi, f}$ are residuated lattices, $\psi$ is a lattice and monoid isomorphism, and the divisions are definable by the order and multiplication, we get that $\psi$ is a residuated-lattice isomorphism.
\end{proof}

%%%%%%%%%%%%%%%%%%%%%%%%%%%%%%%%%%%%%%%%%%%%%%%%%%%%%%%%%%%%%%%%%%%%%%

%%%%%%%%%%%%%%%%%%%%%%%%%%%%%%%%%%%%%%%%%%%%%%%%%%%%%%%%%%%%%%%%%%%%%%

%%%%%%%%%%%%%%%%%%%%%%%%%%%%%%%%%%%%%%%%%%%%%%%%%%%%%%%%%%%%%%%%%%%%%%
%% Acknowledgments (Optional)
%%%%%%%%%%%%%%%%%%%%%%%%%%%%%%%%%%%%%%%%%%%%%%%%%%%%%%%%%%%%%%%%%%%%%%
%\section{Declarations}

%\subsection*{Data availability}
%Data sharing not applicable to this article as datasets were neither generated nor analysed.

%\subsection*{Compliance with ethical standards}
%The authors declare that they have no conflict of interest.

%\subsection*{Competing Interests}
%The first author belongs to the Editorial Board of Algebra Universalis.
%%%%%%%%%%%%%%%%%%%%%%%%%%%%%%%%%%%%%%%%%%%%%%%%%%%%%%%%%%%%%%%%%%%%%%
%% BIBLIOGRAPHY
%%%%%%%%%%%%%%%%%%%%%%%%%%%%%%%%%%%%%%%%%%%%%%%%%%%%%%%%%%%%%%%%%%%%%%

%%%%%%%%%%%%%%%%%%%%%%%%%%%%%%%%%%%%%%%%%%%%%%%%%%%%%%%%%%%%%%
\end{document}